\documentclass[final,leqno]{siamltex}

\usepackage{epsfig}
\usepackage{float}
\usepackage{verbatim}
\usepackage{graphicx}
\usepackage{subfigure}
\usepackage{mathrsfs}
\usepackage{amsmath,amssymb}
\usepackage{extarrows}
\usepackage{color}
\usepackage{booktabs} 
\usepackage{multirow} 

\newcommand{\bea}{\begin{eqnarray}}
\newcommand{\eea}{\end{eqnarray}}
\newcommand{\beq}{\begin{equation}}
\newcommand{\eeq}{\end{equation}}
\newcommand{\beas}{\begin{eqnarray*}}
\newcommand{\eeas}{\end{eqnarray*}}

\title{Global behavior of temporal discretizations for Volterra integrodifferential equations with certain nonsmooth kernels}

\author{Wenlin Qiu
\thanks{School of Mathematics and Statistics, Hunan Normal University, Changsha, Hunan 410081, P. R. China. (Email: {\tt qwllkx12379@163.com}).}
        }

\begin{document}

\maketitle

\begin{abstract} In this work, the z-transform is presented to analyze time-discrete solutions for Volterra integrodifferential equations (VIDEs) with nonsmooth multi-term kernels in the Hilbert space, and this class of continuous problem was first considered and analyzed by Hannsgen and Wheeler (SIAM J Math Anal 15 (1984) 579-594). This work discusses three cases of kernels $\beta_q(t)$ included in the integrals for the multi-term VIDEs, from which we use corresponding numerical techniques to approximate the solution of multi-term VIDEs in different cases. Firstly, for the case of $\beta_1(t), \beta_2(t) \in \mathrm{L}_1(\mathbb{R}_+)$, the Crank-Nicolson (CN) method and interpolation quadrature (IQ) rule are applied to time-discrete solutions of the multi-term VIDEs; secondly, for the case of $\beta_1(t)\in \mathrm{L}_1(\mathbb{R}_+)$ and $\beta_2(t)\in \mathrm{L}_{1,\text{loc}}(\mathbb{R}_+)$, second-order backward differentiation formula (BDF2) and second-order convolution quadrature (CQ) are employed to discretize the multi-term problem in the time direction; thirdly, for the case of $\beta_1(t), \beta_2(t)\in \mathrm{L}_{1,\text{loc}}(\mathbb{R}_+)$, we utilize the CN method and trapezoidal CQ (TCQ) rule to approximate temporally the multi-term problem. Then for the discrete solution of three cases, the long-time global stability and convergence are proved based on the z-transform and certain appropriate assumptions. Furthermore, the long-time estimate of the third case is confirmed by the numerical tests.
\end{abstract}



\pagestyle{myheadings}
\thispagestyle{plain}
\markboth{}{Global behavior of temporal discretizations for VIDEs}

\section{Introduction}
In this work, consider the following VIDEs with multi-term nonsmooth kernels
\begin{equation}\label{eq1.1}
	\begin{split}
            u'(t)   + \int_{0}^{t} \mathbf{B}(t-s)u(s)ds &= f(t), \quad t>0, \\
             u(0) & = u_0,
    \end{split}
\end{equation}
in which we assume that
\begin{equation}\label{eq1.2}
	\begin{split}
           \mathbf{B}(t) = \sum\limits_{q=1}^{m} \beta_q(t) \mathbf{B}_q, \quad \int_0^1 \beta_q(t) dt < \infty, \quad 0\leq \beta_q(\infty) < \beta_q(0^+) \leq \infty,
    \end{split}
\end{equation}
where $\mathbf{B}_q$ is densely positive self-adjoint linear operator, defined in a dense subspace $\mathcal{D}(\mathbf{B}_q) \in \mathrm{H}$, where $\mathrm{H}$ indicates the Hilbert space, and $\beta_q(t)$ is assumed to be real-valued and positive definite on $(0,\infty)$ with $1\leq q \leq m <\infty$. The problem \eqref{eq1.1} can be found in some valuable applications, such as electrodynamics, continuum mechanics, thermodynamics, the population biology and so on, see \cite{Pruss,Renardy} and references therein for more details.

In fact, for the theoretical and numerical researches of problem \eqref{eq1.1}, many scholars have developed some excellent works regarding long-time behavior of solutions. Hannsgen and Wheeler \cite{Hannsgen1} first considered the asymptotic behavior of the solution of problem \eqref{eq1.1} with completely monotonic kernels, and established the following weighted estimates
\begin{equation*}
	\begin{split}
              \int_{0}^{\infty} \rho(t)\| u(t) \| dt \leq  \mathcal{C} \| u_0 \|,
    \end{split}
\end{equation*}
where $\rho(t)$ is a weight function, $\mathcal{C}$ is a positive constant independent of $u(t)$, and $\|\cdot\|$ indicates the norm in $\mathrm{H}$. Then, Noren \cite{Noren1,Noren2} extended the completely monotonic kernels of \eqref{eq1.2} to the convex kernels, which generalized the results in \cite{Hannsgen1}. Recently, based on the theoretical framework of \cite{Hannsgen1}, Xu utilized the backward Euler method and Lubich's first-order CQ rule to prove the weighted asymptotic stability \cite{Xu1} and weighted asymptotic convergence \cite{Xu2}. Subsequently, by employing the BDF2 method and Lubich's second-order CQ rule, Xu deduced the weighted asymptotic stability \cite{Xu3} and weighted asymptotic convergence \cite{Xu4} regarding second-order time-discrete schemes.

\vskip 1mm
Without loss of generality, we below consider the problem \eqref{eq1.1} with two-term nonsmooth kernels, that is
\begin{equation}\label{eq1.3}
	\begin{split}
           \frac{\partial u}{\partial t}  + \sum\limits_{q=1}^{2}\int_{0}^{t} \beta_q(t-s) \mathbf{B}_q u(s)ds &= f(t), \quad t>0, \\
             u(0) & = u_0,
    \end{split}
\end{equation}
from which, denote $\Gamma(\cdot)$ as the Gamma function, and the positive-type kernels $\beta_q(t)$, $q=1,2$ are main assumed to satisfy the following three cases, i.e., \par
\textbf{Case I:} $\beta_q(t)\in \mathrm{L}_1(\mathbb{R}_+)$, $q=1,2$; \par
\textbf{Case II:} $\beta_1(t)\in \mathrm{L}_1(\mathbb{R}_+)$ and $\beta_2(t) = \frac{t^{\alpha-1}}{\Gamma(\alpha)}\in \mathrm{L}_{1,\text{loc}}(\mathbb{R}_+)$ with $0<\alpha<1$; \par
\textbf{Case III:} $\beta_q(t) = \frac{t^{\alpha_q-1}}{\Gamma(\alpha_q)}$, $q=1,2$ with $0<\alpha_q<1$.
Hannsgen and Wheeler \cite{Hannsgen1} pointed out that problem \eqref{eq1.3} arises in a linear model for heat flow in a rectangular, orthotropic
material with memory in which the axes of orthotropy are parallel to the edges of the rectangle, see \cite{Carslaw,MacCamy}. Due to the practical applications of \eqref{eq1.3}, some numerical studies were considered to solve this problem with kernels of \textbf{Case III}, for instance, Hu et al. \cite{Hu} constructed and discussed a backward Euler finite difference scheme, and Qiu et al. \cite{Qiu1} further developed and analyzed a BDF2 finite difference scheme. After that, Qiu et al. \cite{Qiu2} designed an exponentially convergent Sinc approach for approximating \eqref{eq1.3}. Cao et al. \cite{Cao} considered a localized meshless method for problem \eqref{eq1.3}. Recently, Qiu \cite{Qiu3} utilized the product integration rule to formulate and analyze an accurate second-order scheme for the problem of type \eqref{eq1.3}, based on the temporal graded meshes. Although much work has existed on solving this type of problem, it is still under development. These studies motivate and inspire our following work.

\vskip 0.2mm
The main purpose of this paper is to use the z-transform to analyze and discuss the long-time behavior of the time-discrete solution of problem \eqref{eq1.3}, where the kernels in the integral terms are real-valued and positive definite instead of completely monotonic kernels, which relaxes the conditions of \cite{Xu1,Xu2,Xu3,Xu4}. At first, we employ the Laplace transform and the Paserval relation to prove the stability of continuous problem \eqref{eq1.3} in the norm $\|\cdot\|_{L^2(H^{m,s_0})}$ (see Theorem \ref{theorem2.5}). Then, aiming at the kernels of \textbf{Case I}--\textbf{Case III}, we develop three kinds of techniques to solve the problem \eqref{eq1.3}: (i) for the first case, we apply the CN method and IQ rule to time-discrete solutions of \eqref{eq1.3}; (ii) for the first case, we apply the CN method and IQ rule to obtain time-discrete solutions of \eqref{eq1.3}; second, the BDF2 method and mixed IQ-CQ rule are used to approximate temporally \eqref{eq1.3}; third, the CN method and TCQ rule are applied to the temporal discretization of \eqref{eq1.3}, from which, three classes of numerical schemes are analyzed by the z-transform, and we establish their long-time global stability and error estimates. Finally, for the third case, we construct fully discrete schemes by spatial finite difference approximations, and numerical experiments are carried out to verify our theoretical results (see Theorem \ref{theorem6.4}).

\vskip 0.2mm
The remainder of this work is organized as follows. In Section \ref{section2}, some helpful notations and some properties of Laplace transform are given, and the stability of continuous problem \eqref{eq1.3} is constructed. In Section \ref{section3}, some significant properties regarding the z-transform are given and deduced. Section \ref{section4} devotes to the establishment and analysis of the Crank-Nicolson IQ scheme. In Section \ref{section5}, the BDF2 IQ-CQ scheme is formulated and discussed. Section \ref{section6} designs and analyzes a Crank-Nicolson TCQ scheme. Then, numerical examples are provided in Section \ref{section7} to validate the theoretical estimates. Lastly, the brief remarks are concluded in Section \ref{section8}.

\vskip 0.2mm
 Throughout this article, $\mathcal{C}$ denotes a positive constant that is independent of the spatial and temporal step sizes, and may be not necessarily the same on each occurrence.

\section{Preliminaries}\label{section2}
In this section, we shall introduce some notations and lemmas for further analysis.

\subsection{Some notations}\label{subsection2.1}
First, we present some helpful notations. Denote the inner product for $L^2(\Omega)$ space, that is
   \begin{equation*}
    \begin{array}{cc}
          ( w, v ) = \int_{\Omega} w \overline{v} d\Omega,
   \end{array}
   \end{equation*}
and for $0\leq m \in \mathbb{Z}_+$, let $H^m(\Omega)$ be the Sobolev space on $\Omega$ with the norm
   \begin{equation*}
    \begin{array}{cc}
         \|w\|_{H^m} = \sqrt{\sum_{0\leq \tilde{\sigma}_1+\tilde{\sigma}_2 \leq m}\bigg\|\frac{\partial^{\tilde{\sigma}_1+\tilde{\sigma}_2}w}{\partial x^{\tilde{\sigma}_1}\partial y^{\tilde{\sigma}_2}}\bigg\|^{2}},
   \end{array}
   \end{equation*}
from which $\|\cdot\|$ denotes the standard $L^2$ norm, or written as $\|\cdot\|_{H^0}$ for the convenience. Also, we set $H^1_0=H^1\cap\{w|w=0 \; \text{on} \; \partial\Omega\}$. Then for $w: [0,T]\longrightarrow H^m(\Omega)$, denote
   \begin{equation*}
    \begin{array}{cc}
      \|w\|_{L^2(H^m)}=\Big(\int_0^T\|w\|_{m}^{2}dt\Big)^{1/2} \quad \text{with} \quad  \|w\|_{L^2(H^0)}= \|w\|_{L^2(L^2)},
    \end{array}
   \end{equation*}
where $\|w\|_{m}=\|w\|_{H^m}$, and for $w: [0,\infty)\longrightarrow H^m(\Omega)$, we denote the norm by
   \begin{equation*}
    \begin{array}{cc}
       \|w\|_{L^2(H^{m,s_0})}=\Big(\int_0^\infty\|w\|_{m}^{2}e^{-2s_0t}dt\Big)^{1/2}, \qquad s_0>0.
    \end{array}
   \end{equation*}

\subsection{The Laplace transform}\label{subsection2.2}
For the further discussion, we introduce the Laplace transform of a function on $[0,\infty)$, that is
\begin{equation}\label{eq2.1}
    \begin{array}{cc}
      \mathscr{L}[u] := \hat{u}(s) = \int_{0}^{\infty} e^{-st}u(t) dt,
    \end{array}
   \end{equation}
and its certain properties are reviewed as follows.

\begin{lemma}\label{lemma2.1} \cite{Yan} (Paserval relation) For $s_0>0$, we have
\begin{eqnarray*}
\begin{aligned}
         \int_{-\infty}^{+\infty} \hat{u}(s_0+\mathrm{i}\eta) \hat{v}(s_0-\mathrm{i}\eta) d\eta
         = 2\pi  \int_{0}^{+\infty} e^{-2s_0t} u(t)v(t) dt,
\end{aligned}
\end{eqnarray*}
and that
\begin{eqnarray*}
\begin{aligned}
         \int_{-\infty}^{+\infty} \Big\|\hat{u}(s_0+\mathrm{i}\eta)\Big\|^2_{\widetilde{\mathrm{H}}}d\eta
         = 2\pi \int_{0}^{+\infty} e^{-2s_0t} \|u(t)\|^2 dt,
\end{aligned}
\end{eqnarray*}
from which, $\mathrm{i}^2=-1$ and $\widetilde{\mathrm{H}}$ is the complexification of space $\mathrm{H}=L^2(\Omega)$.
\end{lemma}

Then, we present the following lemma (regarding positive-type kernel), which will often be used in this work.
\begin{lemma}\label{lemma2.2} \cite{Nohel} $\beta(t)\in \mathrm{L_{1,loc}}(0,\infty)$ is positive-type if and only if $\Re \left(\hat{\beta}(s)\right) \geq 0$ for $s\in\Phi:=\{s\in\varsigma, \Re (s) >0\}$, from which, $\hat{\beta}$ indicates the Laplace transform of the kernel $\beta(t)$.
\end{lemma}

\vskip 1mm
Next, we give the following lemma under the inhomogeneous case $f\neq 0$.
\begin{lemma}\label{lemma2.3} If $s_0>0$, and
\begin{eqnarray*}
\begin{aligned}
         \int_{0}^{+\infty} w'(t)w(t) e^{-2s_0t} dt \leq \int_{0}^{+\infty} f(t)w(t) e^{-2s_0t} dt,
\end{aligned}
\end{eqnarray*}
then we have
\begin{eqnarray*}
\begin{aligned}
         \int_{0}^{+\infty} w^2(t) e^{-2s_0t} dt \leq   s_0^{-1} w^2(0) + s_0^{-2}\int_{0}^{+\infty} f^2(t)e^{-2s_0t} dt.
\end{aligned}
\end{eqnarray*}
\end{lemma}
\begin{proof} First, we use the integration by parts to get
\begin{eqnarray*}
\begin{aligned}
         \int_{0}^{+\infty}  & e^{-2s_0t}[w^2(t)-w^2(0)] dt + \lim\limits_{t\rightarrow \infty} \left( \frac{1}{2s_0} e^{-2s_0t}w^2(t)\right) \\
         &= \frac{1}{s_0} \int_{0}^{+\infty}  e^{-2s_0t} w^{(1)}(t)w(t)dt \leq  \frac{1}{s_0} \int_{0}^{+\infty}  e^{-2s_0t} f(t)w(t)dt  \\
         & \leq \frac{1}{s_0} \int_{0}^{+\infty}  e^{-2s_0t} \left( \frac{f^2(t)}{2s_0} + \frac{s_0w^2(t)}{2} \right) dt,
\end{aligned}
\end{eqnarray*}
which naturally can obtain
\begin{eqnarray*}
\begin{aligned}
       \frac{1}{2} \int_{0}^{+\infty}  e^{-2s_0t}w^2(t) dt \leq \int_{0}^{+\infty}  e^{-2s_0t}w^2(0) dt
                + \frac{1}{2s_0^2} \int_{0}^{+\infty}  e^{-2s_0t}f^2(t)dt.
\end{aligned}
\end{eqnarray*}
The proof is finished.
\end{proof}

\vskip 2mm
Further, the following corollary is established by Lemma \ref{lemma2.3}.
\begin{corollary}\label{corollary2.4} Let $s_0>0$. If
\begin{eqnarray*}
\begin{aligned}
         \int_{0}^{+\infty} \left(\frac{du(t)}{dt},u(t)\right) e^{-2s_0t} dt \leq \int_{0}^{+\infty} (f(t),u(t)) e^{-2s_0t} dt,
\end{aligned}
\end{eqnarray*}
then we yield
\begin{eqnarray*}
\begin{aligned}
         \int_{0}^{+\infty} \|u(t)\|^2 e^{-2s_0t} dt \leq  s_0^{-1}\|u(0)\|^2 + s_0^{-2}\int_{0}^{+\infty} \|f(t)\|^2e^{-2s_0t} dt.
\end{aligned}
\end{eqnarray*}
\end{corollary}

\vskip 0.2mm
Then based on above analyses, we get the following theorem.
\begin{theorem} \label{theorem2.5} (Stability) Let $u$ be the solution of the continuous problem \eqref{eq1.3}. If $s_0>0$ and $\beta_q(t)\in \mathrm{L_{1,loc}}(0,\infty)$, $q=1,2$, then we obtain
      \begin{equation*}
     \begin{split}
           \|u\|^2_{L^2(H^{0,s_0})} \leq  s_0^{-1} \|u_0\|^2 + s_0^{-2}\|f\|^2_{L^2(H^{0,s_0})}.
     \end{split}
     \end{equation*}
\end{theorem}
\begin{proof} First, we apply the Laplace transform to \eqref{eq1.3}, then
      \begin{equation}\label{eq2.2}
     \begin{split}
             \frac{\widehat{d u}(s)}{dt}  + \sum\limits_{q=1}^{2} \hat{\beta}_q(s) \mathbf{B}_q \hat{u}(s) = \hat{f}(s), \quad s\in \Phi.
     \end{split}
     \end{equation}
Taking the inner product of \eqref{eq2.2} with $\hat{u}(s)$ that
      \begin{equation*}
     \begin{split}
             \left(\frac{\widehat{d u}(s)}{dt},\hat{u}(s)\right) + \sum\limits_{q=1}^{2} \hat{\beta}_q(s) \left(\mathbf{B}_q \hat{u}(s),\hat{u}(s)\right) = \left(\hat{f}(s),\hat{u}(s)\right),
     \end{split}
     \end{equation*}
from which, we let $s=s_0+\mathrm{i}\eta$, and take the real part of above formula, then
      \begin{equation*}
     \begin{split}
             \Re\left(\frac{\widehat{d u}(s_0+\mathrm{i}\eta)}{dt},\hat{u}(s_0+\mathrm{i}\eta)\right) &+ \sum\limits_{q=1}^{2} \Re\left(\hat{\beta}_q(s_0+\mathrm{i}\eta) \right) \left(\mathbf{B}_q \hat{u}(s_0+\mathrm{i}\eta),\hat{u}(s_0+\mathrm{i}\eta)\right) \\
              & \leq \Re\left(\hat{f}(s_0+\mathrm{i}\eta),\hat{u}(s_0+\mathrm{i}\eta)\right).
     \end{split}
     \end{equation*}
Further, we use Lemma \ref{lemma2.2} and the positivity of the linear operator $\mathbf{B}_q$ to get
      \begin{equation}\label{eq2.3}
     \begin{split}
             \Re\left(\frac{\widehat{d u}(s_0+\mathrm{i}\eta)}{dt},\hat{u}(s_0+\mathrm{i}\eta)\right) \leq \Re\left(\hat{f}(s_0+\mathrm{i}\eta),\hat{u}(s_0+\mathrm{i}\eta)\right).
     \end{split}
     \end{equation}
Next, assuming that $\{\tilde{\varphi}_l\}_{l=1}^{\infty}$ is an orthogonal basis in $\mathrm{H}$ such that $v(t)\in \mathrm{H}$, then we obtain a representation $v(t) = \sum_{l=1}^{\infty} (v(t), \tilde{\varphi}_l)  \tilde{\varphi}_l$. Therefore, employing Lemma \ref{lemma2.1}, we have
      \begin{equation}\label{eq2.4}
     \begin{split}
            \int_{-\infty}^{+\infty} & \left(\frac{\widehat{d u}}{dt},\hat{u}\right) (s_0+\mathrm{i}\eta) d\eta
            = \sum\limits_{l=1}^{\infty} \int_{-\infty}^{+\infty} \left(\frac{\widehat{d u}}{dt},\tilde{\varphi}_l\right)(s_0+\mathrm{i}\eta)   \left(\hat{u},\tilde{\varphi}_l\right)(s_0-\mathrm{i}\eta)d\eta \\
            & = 2\pi \sum\limits_{l=1}^{\infty}  \int_{0}^{+\infty} e^{-2s_0 t} \left(\frac{d u}{dt},\tilde{\varphi}_l\right) \left(u,\tilde{\varphi}_l\right) dt
            = 2\pi  \int_{0}^{+\infty} e^{-2s_0 t} \left(\frac{d u}{dt},u\right)dt,
     \end{split}
     \end{equation}
and
      \begin{equation}\label{eq2.5}
     \begin{split}
            \int_{-\infty}^{+\infty} & \left(\hat{f},\hat{u}\right) (s_0+\mathrm{i}\eta) d\eta
            = \sum\limits_{l=1}^{\infty} \int_{-\infty}^{+\infty} \left(\hat{f},\tilde{\varphi}_l\right)(s_0+\mathrm{i}\eta)   \left(\hat{u},\tilde{\varphi}_l\right)(s_0-\mathrm{i}\eta)d\eta \\
            & = 2\pi  \int_{0}^{+\infty} e^{-2s_0 t} \left(f,u\right)dt.
     \end{split}
     \end{equation}
By substituting \eqref{eq2.4} and \eqref{eq2.5} into \eqref{eq2.3}, we have
      \begin{equation}\label{eq2.6}
     \begin{split}
            \int_{0}^{+\infty} e^{-2s_0 t} \left(\frac{d u}{dt},u\right)dt \leq \int_{0}^{+\infty} e^{-2s_0 t} \left(f,u\right)dt,
     \end{split}
     \end{equation}
from which, utilizing Corollary \ref{corollary2.4}, we complete the proof.
\end{proof}

\section{The z-transform}\label{section3}
 Here, we first present the z-transform regarding a real sequence or H-value sequence $\{g_n\}_{n=0}^{\infty}$, namely
      \begin{equation}\label{eq3.1}
     \begin{split}
          G(z):=   \mathcal{Z}(\{g_n\}_{n=0}^{\infty})(z) = \sum\limits_{n=0}^{\infty} g_n z^{-n}.
     \end{split}
     \end{equation}

\vskip 0.2mm
Then, some significant properties of the z-transform are introduced as follows.

\begin{proposition}\label{proposition3.1}\cite{Jury} (Convolution theorem) Let $(g\circ d)(n):=\sum\limits_{j=0}^{n}g_{n-j}d_j$.
If G(z) and $D(z)$ are the z-transform of sequences $\{g_n\}_{n=0}^{\infty}$ and $\{d_n\}_{n=0}^{\infty}$, respectively, then
    \begin{equation}\label{eq3.2}
     \begin{split}
         \mathcal{Z}\Big((g\circ d)(n)\Big)(z) = G(z) D(z).
     \end{split}
     \end{equation}
\end{proposition}

\begin{proposition}\label{proposition3.2}\cite{Jury}
Let G(z) and $D(z)$ be the z-transform of the sequences $\{g_n\}_{n=0}^{\infty}$ and $\{d_n\}_{n=0}^{\infty}$, respectively. If
    \begin{equation}\label{eq3.3}
     d_n =
     \begin{cases}
         0, & n<1, \\
         g_{n-1}, & n\geq 1,
     \end{cases}
     \end{equation}
then we have $D(z)=G(z)/z$.
\end{proposition}

\vskip 0.2mm
\begin{proposition}\label{proposition3.3}\cite{Jury} (Inverse z-transform) Assume that $G(z)$ is the z-transform of $\{g_m\}_{m=0}^{\infty}$ and the contour $\Gamma_0$ encloses all poles of $G(z)$, then
    \begin{equation}\label{eq3.4}
     \begin{split}
         g_m = \frac{1}{2\pi\mathrm{i}} \oint_{\Gamma_0} G(z) z^{m-1} dz,  \quad  m = 0,1,2, \cdots.
     \end{split}
     \end{equation}
\end{proposition}

\vskip 0.2mm
Then for $\forall \epsilon>0$, let $\delta(t-n\epsilon)$ be the unit pulse function at $t=n\epsilon$ and denote
    \begin{equation}\label{eq3.5}
     \begin{split}
          g_{\epsilon}(t) = \sum\limits_{n=0}^{\infty} g_n \delta(t-n\epsilon).
     \end{split}
     \end{equation}
By applying the Laplace transform to \eqref{eq3.5}, we can get
    \begin{equation}\label{eq3.6}
     \begin{split}
          G_{\epsilon}(s) := \mathscr{L}\left[g_{\epsilon}(t)\right] = \sum\limits_{n=0}^{\infty} g_n e^{-n\epsilon s},
     \end{split}
     \end{equation}
from which we use the fact that $\mathscr{L}[\delta(t-n\epsilon)]=\int_{0}^{\infty}\delta(t-n\epsilon)e^{-st}dt=e^{-n\epsilon s}$.

\vskip 0.2mm
Consequently, the connection between the Laplace transform and z-transform can be established by
    \begin{equation}\label{eq3.7}
     \begin{split}
           G(z) = G_{\epsilon}(s)\big|_{s=(1/\epsilon)\ln z} = \mathscr{L}\left[g_{\epsilon}(t)\right]\big|_{z=e^{\epsilon s}}.
     \end{split}
     \end{equation}

\vskip 0.2mm
Next, for the further analysis, we define $\delta_{\epsilon}(t)=\int_{0}^{\infty}\delta(t-n\epsilon)$ and denote
    \begin{equation}\label{eq3.8}
     \begin{split}
           G_{\epsilon}(s) = \mathscr{L}\left[g_{\epsilon}(t)\right] = \mathscr{L}\left[ g(t) \delta_{\epsilon}(t) \right],
     \end{split}
     \end{equation}
from which, $g(t)$ is expressed as a continuous function over $[0,\infty)$, such that $g(n \epsilon)=g_n$ for $n=0,1,2,\cdots$. Then we can obtain the following lemma.

\vskip 0.5mm
\begin{lemma}\label{lemma3.4} Let two sampled functions $G_{\epsilon}(s)$ and $H_{\epsilon}(s)$ be Laplace transforms of $g_{\epsilon}(t)$ and $H_{\epsilon}(t)$, respectively. Then for $c>0$, it holds that
    \begin{equation}\label{eq3.9}
     \begin{split}
           \sum\limits_{m=0}^{\infty} (g_m, h_m) e^{-m\epsilon s} &=  \mathscr{L}\left[ \big(g(t), h(t) \big) \cdot \delta_{\epsilon}(t) \right] \\
          &  = \frac{\epsilon}{2\pi\mathrm{i}}\int_{c-\mathrm{i}\pi/\epsilon}^{c+\mathrm{i}\pi/\epsilon} \Big(G_{\epsilon}(\vartheta), \overline{H_{\epsilon}(s-\vartheta)} \Big)_{\widetilde{\mathrm{H}}} d\vartheta,
     \end{split}
     \end{equation}
in which $\bar{z}$ indicates the complex conjugate of $z$.
\end{lemma}
\begin{proof} First, utilizing \cite[(1.129)]{Jury}, and assuming $g(t)$ contains no impulses and is initially zero, then we have
    \begin{equation*}
     \begin{split}
          & \sum\limits_{m=0}^{\infty} (g_m, h_m) e^{-m\epsilon s} =  \mathscr{L}\left[ \big(g(t), h(t) \big) \cdot \delta_{\epsilon}(t) \right]
         \\
          & = \frac{1}{2\pi\mathrm{i} } \int_{w_0-\mathrm{i}\infty}^{w_0+\mathrm{i}\infty} \left\{ \mathscr{L}\left[ \big(g(t), h(t) \big) \right](w) \right\} \; \mathscr{L}\left[ \delta_{\epsilon}(t) \right](s-w) dw
     \end{split}
     \end{equation*}
with $\Re(s)>w_0>\max\left\{\text{real part of poles of} \; \mathscr{L} \big(g(t), h(t) \big) \right\}$. Then use the inverse Laplace transform, we get
    \begin{equation*}
     \begin{split}
           \mathscr{L}\left[ \big(g(t), h(t) \big) \right](w) &= \int_{0}^{\infty} \left(  \frac{1}{2\pi\mathrm{i} } \int_{c_0-\mathrm{i}\infty}^{c_0+\mathrm{i}\infty} g(\vartheta) e^{\vartheta t}d\vartheta, h(t)  \right) e^{wt} dt \\
           & = \frac{1}{2\pi\mathrm{i} } \int_{c_0-\mathrm{i}\infty}^{c_0+\mathrm{i}\infty} \left(  \hat{g}(\vartheta), \overline{\hat{h}(w-\vartheta)} \right)_{\widetilde{\mathrm{H}}} d\vartheta.
     \end{split}
     \end{equation*}
Using above two formulas and changing the order of integration regarding $w$ and $\vartheta$, we further obtain
    \begin{equation*}
     \begin{split}
           \mathscr{L}\left[ \big(g(t), h(t) \big) \cdot \delta_{\epsilon}(t) \right]
           & = \frac{1}{2\pi\mathrm{i} } \int_{c_0-\mathrm{i}\infty}^{c_0+\mathrm{i}\infty}  \frac{1}{2\pi\mathrm{i} } \int_{w_0-\mathrm{i}\infty}^{w_0+\mathrm{i}\infty} \left(  \hat{g}(\vartheta), \overline{\hat{h}(w-\vartheta)} \right)_{\widetilde{\mathrm{H}}} \sum\limits_{l=0}^{\infty}e^{-l\epsilon(s-w)} dw d\vartheta \\
           & = \frac{1}{2\pi\mathrm{i} } \int_{c_0-\mathrm{i}\infty}^{c_0+\mathrm{i}\infty} \left( \hat{g}(\vartheta),  \overline{  \frac{1}{2\pi\mathrm{i} } \int_{w_0-\mathrm{i}\infty}^{w_0+\mathrm{i}\infty} \hat{h}(w-\vartheta) \frac{dw}{1-e^{-\epsilon(s-w)}}}   \right)_{\widetilde{\mathrm{H}}} d\vartheta \\
           & = \frac{1}{2\pi\mathrm{i} } \int_{c_0-\mathrm{i}\infty}^{c_0+\mathrm{i}\infty} \left( \hat{g}(\vartheta), \overline{H_\epsilon(s-\vartheta)}\right)_{\widetilde{\mathrm{H}}} d\vartheta.
     \end{split}
     \end{equation*}
Note that $H_\epsilon(s-\vartheta)$ is only a function of $e^{\epsilon(s-\vartheta)}$. In above equation, we divide the range of integration into intervals of length $\tilde{w}_0=2\pi/\epsilon$, therefore we yield
    \begin{equation*}
     \begin{split}
           \mathscr{L}\left[ \big(g(t), h(t) \big) \cdot \delta_{\epsilon}(t) \right](s)
           & = \frac{1}{2\pi\mathrm{i} } \sum\limits_{m=-\infty}^{+\infty} \int_{c_0+(m-1/2)\mathrm{i}\tilde{w}_0}^{c_0+(m+1/2)\mathrm{i}\tilde{w}_0} \left( \hat{g}(\vartheta), \overline{H_\epsilon(s-\vartheta)}\right)_{\widetilde{\mathrm{H}}} d\vartheta.
     \end{split}
     \end{equation*}
Here, altering $\vartheta$ to $\vartheta+\mathrm{i}m\tilde{w}_0$ and since $H_\epsilon(s-\vartheta)$ is only a function of $e^{-\epsilon(s-\vartheta)}$ with $e^{ \mathrm{i}m\tilde{w}_0\epsilon}=1$, we obtain
    \begin{equation*}
     \begin{split}
           \mathscr{L}\left[ \big(g(t), h(t) \big) \cdot \delta_{\epsilon}(t) \right](s)
           & = \frac{1}{2\pi\mathrm{i} } \sum\limits_{m=-\infty}^{+\infty} \int_{c_0-\mathrm{i}\tilde{w}_0/2}^{c_0+\mathrm{i}\tilde{w}_0/2} \left( \hat{g}(\vartheta+\mathrm{i}m\tilde{w}_0), \overline{H_\epsilon(s-\vartheta)}\right)_{\widetilde{\mathrm{H}}} d\vartheta.
     \end{split}
     \end{equation*}
Then from \cite[(1.141)-(1.142)]{Jury}, above equation can be simplified as
    \begin{equation*}
     \begin{split}
           \mathscr{L}\left[ \big(g(t), h(t) \big) \cdot \delta_{\epsilon}(t) \right](s)
           & = \frac{\epsilon}{2\pi\mathrm{i}} \int_{c_0-\mathrm{i}\tilde{w}_0/2}^{c_0+\mathrm{i}\tilde{w}_0/2}  \left( G_{\epsilon}(\vartheta), \overline{H_\epsilon(s-\vartheta)}\right)_{\widetilde{\mathrm{H}}} d\vartheta,
     \end{split}
     \end{equation*}
which completes the proof.
\end{proof}

\vskip 2mm
Then from Lemma \ref{lemma3.4}, we replace $e^{\epsilon s}$ and $e^{\epsilon \vartheta}$ by $z$ and $p$, respectively, therefore, $G_{\epsilon}(\vartheta)$ turns into a function of $p$, i.e., $\mathbf{G}(p)$, and $H_{\epsilon}(s-\vartheta)$ becomes the function of $\frac{z}{p}$, i.e., $\mathbf{H}(z/p)$. Naturally, the following lemma holds.

\begin{lemma}\label{lemma3.5}  By a mapping $d\vartheta = (1/(\epsilon e^{\epsilon \vartheta})) dp = (\epsilon^{-1} p^{-1}) dp$, then \eqref{eq3.9} turns into
    \begin{equation}\label{eq3.10}
     \begin{split}
          & \sum\limits_{m=0}^{\infty} (g_m, h_m) z^{-m} =  \mathscr{L}\left[ \big(g(t), h(t) \big) \delta_{\epsilon}(t) \right]  = \frac{1}{2\pi\mathrm{i}}\int_{\Gamma_0} p^{-1} \Big( \mathbf{G}(p), \overline{\mathbf{H}(z/p)} \Big)_{\widetilde{\mathrm{H}}} dp.
     \end{split}
     \end{equation}
\end{lemma}

\vskip 0.2mm
In fact, two lemmas above have extended the case of real-valued sequence \cite{Jury} to that of H-valued sequences, which will be helpful to next analysis. Besides, for our theoretical estimates, a key lemma is introduced as follows.

\vskip 0.5mm
\begin{lemma}\label{lemma3.6}\cite{Lopez-Marcos} If the sequence $\{g_j\}_{j=0}^{\infty}$ is real-valued, satisfying that $G(z)=\sum\limits_{j=0}^{\infty}g_j z^{-j}$ is analytical in $\mathbf{D}=\{z\in \varsigma, \; |z|>1\}$, then for any $0<M \in \mathbb{Z}_+$ and for any $(W_0,W_1,\cdots,W_M)\in R^{M+1}$, it holds that
    \begin{equation*}
     \begin{split}
           \sum\limits_{j=0}^{M} \left( \sum\limits_{p=0}^{j} g_p W_{j-p} \right) W_j \geq 0,
     \end{split}
     \end{equation*}
if and only if $\Re G(z)\geq 0$ for $z\in \mathbf{D}$.
\end{lemma}

\section{Crank-Nicolson IQ method for kernels of Case I}\label{section4}
Here, we shall use the Crank-Nicolson method to approximate \eqref{eq1.3} with non-smooth kernels. First, we consider the kernels of \textbf{Case I}.

In order to approximate the integrals of \eqref{eq1.3} formally to the second order, we employ the IQ rule
      \begin{equation}\label{eq4.1}
     \begin{split}
             Q_{n}^{(q)}(\varphi) = \chi_{n,0}^{(q)} \varphi^0 + \sum\limits_{l=1}^{n} w_{n-l}^{(q)} \varphi^{l}, \quad q=1,2,
     \end{split}
     \end{equation}
see \cite{McLean1}; from which, the following relations hold
      \begin{equation}\label{eq4.2}
     \begin{split}
             &   \chi_{0,0}^{(q)} = 0,  \quad 0\leq \chi_{n,0}^{(q)} = \int_{0}^{k} \beta_q(t_n-\zeta) h(\zeta/k) d\zeta, \quad n\geq 1,  \\
             & 0\leq w_{n-l}^{(q)} \equiv \chi_{n,l}^{(q)} = \int_{-k}^{\min(k,t_{n-l})} \beta_q(t_{n-l}-\zeta) h(\zeta/k) d\zeta, \quad n\geq 1, \quad l\geq 1,
     \end{split}
     \end{equation}
where $k$ is the temporal step size and $h(t)=\max(1-|t|, 0)$. Especially,
      \begin{equation}\label{eq4.3}
     \begin{split}
            w_{n}^{(q)} = \int_{-k}^{\min(k,t_{n})} \beta_q(t_{n}-\zeta) h(\zeta/k) d\zeta, \quad n\geq 0.
     \end{split}
     \end{equation}

Then, it follows from \cite[Lemma 4.8]{McLean1} that
      \begin{equation}\label{eq4.4}
     \begin{split}
             \sum\limits_{n=1}^{N} \left( \sum\limits_{l=1}^{n} w_{n-l} V_{l} \right) V_n = k^{-1} \sum\limits_{n=1}^{N} \left( \sum\limits_{l=1}^{n} \left(k\chi_{n,l}^{(q)} \right) V_{l} \right) V_n \geq 0, \quad N\geq 1,
     \end{split}
     \end{equation}
and defining $\breve{\mu}_j^{(q)}=\int_{t_j}^{t_{j+1}}|\beta_q(\zeta)|d\zeta$, we obtain
      \begin{equation}\label{eq4.5}
     \begin{split}
            \big\| (\beta_q \ast \varphi)(t_n)  - Q_{n}^{(q)}(\varphi) \big\|  & \leq 2  \breve{\mu}_{n-1}^{(q)} \int_{0}^{k}\|\varphi_t(s)\|ds  \\
            & + k\sum\limits_{m=2}^{n} \breve{\mu}_{n-m}^{(q)} \int_{t_{m-1}}^{t_{m}} \|\varphi_{tt}(s)\| ds,
     \end{split}
     \end{equation}
where the convolution $(\beta_q \ast \varphi)(t)=\int_{0}^{t}\beta_q(t-s)\varphi(s)ds$. Also, using the assumptions (\textbf{Case I}) regarding the kernels $\beta_q$, we have
      \begin{equation}\label{eq4.6}
     \begin{split}
            \sum\limits_{n=0}^{\infty} w_{n}^{(q)}  \leq w_0^{(q)} + 4 \int_{0}^{\infty} \beta_q(\zeta) d\zeta  < \infty.
     \end{split}
     \end{equation}

\vskip 0.2mm
Thus for $n\geq 1$ and $q=1,2$, we have
      \begin{equation}\label{eq4.7}
     \begin{split}
          Q_{n-\frac{1}{2}}^{(q)}(\varphi) &= \frac{1}{2}\Big( Q_{n}^{(q)}(\varphi) + Q_{n-1}^{(q)}(\varphi) \Big) = \widetilde{\chi}_n^{(q)} \varphi^0  +  \sum\limits_{l=1}^{n} w_{n-l}^{(q)} \varphi^{l-\frac{1}{2}},
     \end{split}
     \end{equation}
where $\widetilde{\chi}_n^{(q)} = \frac{1}{2}\left( \chi_{n,0}^{(q)} + \chi_{n-1,0}^{(q)} - w_{n-1}^{(q)} \right)$ and $\varphi^{n-\frac{1}{2}} = \frac{1}{2}\left( \varphi^{n} + \varphi^{n-1} \right)$.

\vskip 2mm
Here, the Crank-Nicolson method and \eqref{eq4.7} are applied to problem \eqref{eq1.3}, then
   \begin{equation}\label{eq4.8}
   \begin{split}
      \delta_t U^n  & + \sum\limits_{q=1}^{2} Q_{n-1/2}^{(q)}(\mathbf{B}_qU)   = f^{n-\frac{1}{2}},   \quad   n\geq 1, \\
      & U^0=u_0,
   \end{split}
   \end{equation}
where $\delta_t v^n=\frac{v^n-v^{n-1}}{k}$ and $v^{n-\frac{1}{2}}=(v^{n-1}+v^{n} )/2$ for $n\geq 1$.

\vskip 0.5mm
Then applying the z-transform to \eqref{eq4.8}, we yield
   \begin{equation}\label{eq4.9}
   \begin{split}
     \sum\limits_{n=1}^{\infty} \frac{U^n-U^{n-1}}{k} z^{-n}  & +
        \sum\limits_{q=1}^{2} \sum\limits_{n=1}^{\infty} \left(  \widetilde{\chi}_n^{(q)} \mathbf{B}_qU^0  +  \sum\limits_{l=1}^{n} w_{n-l}^{(q)} \mathbf{B}_qU^{l-\frac{1}{2}}  \right) z^{-n} \\
        &= \sum\limits_{n=1}^{\infty}\frac{f^n+f^{n-1}}{2} z^{-n},
   \end{split}
   \end{equation}
and utilizing Propositions \ref{proposition3.1} and \ref{proposition3.2}, we further get
   \begin{equation}\label{eq4.10}
   \begin{split}
        & \frac{2(1-z^{-1})}{1+z^{-1}} \widetilde{U}(z) +  k \sum\limits_{q=1}^{2}\widetilde{w}^{(q)}(z) \mathbf{B}_q\widetilde{U}(z)
       = k\widetilde{F}(z) - \frac{k}{1+z^{-1}} f^0 \\
       & + \frac{2}{1+z^{-1}} U^0 +  \frac{2k}{1+z^{-1}}  \sum\limits_{q=1}^{2} \left( \frac{1}{2}\widetilde{w}^{(q)}(z)-\widetilde{\chi}^{(q)}(z) \right) \mathbf{B}_qU^0,
   \end{split}
   \end{equation}
from which,
   \begin{equation}\label{eq4.11}
   \begin{split}
       & \widetilde{U}(z) = U^0 + U^1 z^{-1} + \cdots + U^n z^{-n} + \cdots, \\
       & \widetilde{F}(z) = f^0 + f^1 z^{-1} + \cdots + f^n z^{-n} + \cdots, \\
       & \widetilde{w}^{(q)}(z) = w_0^{(q)} + w_1^{(q)} z^{-1} + \cdots + w_n^{(q)} z^{-n} + \cdots, \quad q =1,2, \\
       & \widetilde{\chi}^{(q)}(z) = \widetilde{\chi}_1^{(q)} z^{-1} + \widetilde{\chi}_2^{(q)} z^{-2} + \cdots + \widetilde{\chi}_n^{(q)} z^{-n} + \cdots, \quad q =1,2.
   \end{split}
   \end{equation}

\vskip 0.5mm
Then, the following lemma is established by above analyses.
\begin{lemma}\label{lemma4.1} Let the series $\widetilde{w}^{(q)}(z)=\sum_{n=0}^{\infty}w_n^{(q)}z^{-n}$, $q =1,2$, be denoted in \eqref{eq4.11}. Then it holds that $\Re \left(\widetilde{w}^{(q)}(z)\right) \geq 0$.
\end{lemma}
\begin{proof} $(\mathrm{I})$ First, the result \eqref{eq4.4} holds. $(\mathrm{II})$ Then, \eqref{eq4.6} implies that the positive-term series $\sum_{n=0}^{\infty}w_n^{(q)}$ is convergent, and then Abel theorem leads to $\sum_{n=0}^{\infty}w_n^{(q)}z^{-n}$ is convergent in $\mathbf{D}=\{z\in \varsigma, \; |z|>1\}$. Thus, Weierstrass theorem shows that $\sum_{n=0}^{\infty}w_n^{(q)}z^{-n}$ is analytical in $\mathbf{D}$. Combining $(\mathrm{I})$ and $(\mathrm{II})$, the proof is completed by Lemma \ref{lemma3.6}.
\end{proof}

\vskip 2mm
Below we shall deduce the long-time stability of scheme \eqref{eq4.8}, i.e., the following theorem.
\begin{theorem}\label{theorem4.2} Let $U^n$ denoted by \eqref{eq4.8} be the approximate solution of problem \eqref{eq1.3}. Assuming that $\beta_q(t)\in \mathrm{L_1}(\mathbb{R}_+)$, $q=1,2$ and $c>0$, then it holds that
   \begin{equation*}
   \begin{split}
       \|U^n\|^2_A
          \leq  \frac{e^{c}(1+e^{\frac{c}{2}})^4}{(e^{\frac{c}{2}}-1)^6} \Big( k\|f^0\| + 2\|U^0\|  \Big)^2 &  +  \sum\limits_{q=1}^{2} \frac{\|\beta_q\|^2_{\mathrm{L_1}(0,\infty)}(1+e^{\frac{c}{2}})^4k^2}{e^{c}(e^{\frac{c}{2}}-1)^2} \|\mathbf{B}_q U^0\|^2  \\
        &+ \frac{(1+e^{\frac{c}{2}})^4 k^2}{2(e^{c}-1)^2} \left(\sum\limits_{n=0}^{\infty} e^{-\frac{c}{2} n} \|f^n\|\right)^2,
   \end{split}
   \end{equation*}
where the notation $\|U^n\|_A = \sqrt{\sum\limits_{n=0}^{\infty} e^{-cn}\|U^n\|^2}$.
\end{theorem}

\begin{proof} By taking the inner product of \eqref{eq4.10} with $\widetilde{U}(z)$, we get
   \begin{equation*}
   \begin{split}
        & \frac{2(1-z^{-1})}{1+z^{-1}} \left\|\widetilde{U}(z)\right\|^2_{\widetilde{\mathrm{H}}}   + k \sum\limits_{q=1}^{2}\widetilde{w}^{(q)}(z) (\mathbf{B}_q \widetilde{U}(z), \widetilde{U}(z))
       \\
        & = k(\widetilde{F}(z),\widetilde{U}(z))  - \frac{k}{1+z^{-1}} (f^0,\widetilde{U}(z))   + \frac{2}{1+z^{-1}} (U^0,\widetilde{U}(z)) \\
        &   + \frac{2k}{1+z^{-1}}  \sum\limits_{q=1}^{2}\left( \frac{1}{2}\widetilde{w}^{(q)}(z)-\widetilde{\chi}^{(q)}(z) \right) (\mathbf{B}_qU^0,\widetilde{U}(z)).
   \end{split}
   \end{equation*}
Taking the real part of above formula and letting $z=e^{s_0 + \mathrm{i} \eta}$ with any $s_0>0$, then from
   \begin{equation}\label{eq4.12}
   \begin{split}
        & \Re \left(\frac{2(1-z^{-1})}{1+z^{-1}} \right)\Bigl|_{z=e^{s_0 + \mathrm{i} \eta}}  = \frac{2(e^{2s_0}-1)}{ 1 + 2e^{s_0}\cos(\eta) + e^{2s_0} } \geq \frac{2(e^{2s_0}-1)}{(1+e^{s_0})^2}, \\
        & \Re \left(\frac{1}{1+z^{-1}} \right)\Bigl|_{z=e^{s_0 + \mathrm{i} \eta}}  = \frac{1 + e^{-s_0}\cos(\eta)}{1 + 2e^{-s_0}\cos(\eta) + e^{-2s_0}} \leq \frac{1+e^{-s_0}}{(1-e^{-s_0})^2},
   \end{split}
   \end{equation}
   \begin{equation*}
   \begin{split}
        \Xi^{(q)}(z)  \triangleq\frac{2}{1+z^{-1}} & \left( \frac{1}{2}\widetilde{w}^{(q)}(z)-\widetilde{\chi}^{(q)}(z) \right)  = w_0^{(q)} \\
        & + \sum\limits_{n=1}^{\infty}( w_n^{(q)} -\chi_{n,0}^{(q)} ) z^{-n} - \frac{1}{1+z}\chi_{0,0}^{(q)},
   \end{split}
   \end{equation*}
and $\Re\widetilde{w}(z)\geq 0$ (cf. Lemma \ref{lemma4.1}) it follows  that
   \begin{equation}\label{eq4.13}
   \begin{split}
        \frac{2(e^{2s_0}-1)}{(1+e^{s_0})^2}  \left\|\widetilde{U}(z)\right\|_{\widetilde{\mathrm{H}}}
         & \leq  \frac{1+e^{-s_0}}{(1-e^{-s_0})^2} \Big( k\|f^0\| + 2\|U^0\|  \Big) + k \sum\limits_{n=0}^{\infty} e^{-s_0 n} \|f^n\| \\
        & +   k \sum\limits_{q=1}^{2} \left( w_0^{(q)} + e^{-s_0}\sum\limits_{n=1}^{\infty}| w_n^{(q)} -\chi_{n,0}^{(q)} |  \right) \|\mathbf{B}_qU^0\|,
   \end{split}
   \end{equation}
in which the positivity of the operators $\mathbf{B}_1$, $\mathbf{B}_2$ and
   \begin{equation*}
   \begin{split}
      \Re \left(\Xi^{(q)}(z)\right)=w_0^{(q)} + \sum\limits_{n=1}^{\infty}\left( w_n^{(q)} -\chi_{n,0}^{(q)} \right) e^{-ns_0}\cos(n\eta)
   \end{split}
   \end{equation*}
are utilized. Next, from Lemmas \ref{lemma3.4} and \ref{lemma3.5}, we obtain
   \begin{equation*}
   \begin{split}
       \|U^n\|_A^2 = \frac{1}{2\pi\mathrm{i}}\oint_{C_{e^{c/2}}} p^{-1} \Big( \widetilde{U}(p), \overline{\widetilde{U}(e^c/p)} \Big)_{\widetilde{\mathrm{H}}} dp,
   \end{split}
   \end{equation*}
in which $c>0$ and $C_{e^{c/2}}$ is a circle of radius $e^{\frac{c}{2}}$. Then, taking $p=e^{\frac{c}{2}+ \mathrm{i} \eta}$ to yield
   \begin{equation}\label{eq4.14}
   \begin{split}
        \|U^n\|_A^2 & = \frac{1}{2\pi} \int_{0}^{2\pi} \Big( \widetilde{U}(e^{\frac{c}{2}+ \mathrm{i} \eta}), \overline{\widetilde{U}(e^{\frac{c}{2}- \mathrm{i} \eta})} \Big)_{\widetilde{\mathrm{H}}} d\eta   =  \frac{1}{2\pi} \int_{0}^{2\pi} \left\|\widetilde{U}(e^{\frac{c}{2}+ \mathrm{i} \eta})\right\|^2_{\widetilde{\mathrm{H}}}  d\eta.
   \end{split}
   \end{equation}
In addition, we employ \eqref{eq4.2} and \eqref{eq4.3} to yield that
   \begin{equation*}
   \begin{split}
       \sum\limits_{n=1}^{\infty}\left| w_n^{(q)} -\chi_{n,0}^{(q)} \right| &\leq  \sum\limits_{n=1}^{\infty} \int_{-k}^{0}| \beta_q(t_n-\zeta)|\left|h(\zeta/k) \right| d\zeta \\
       & \leq \sum\limits_{n=1}^{\infty} \int_{t_n}^{t_{n+1}}\beta_q(s) ds \leq \|\beta_q\|_{\mathrm{L_1}(0,\infty)},
   \end{split}
   \end{equation*}
and that
   \begin{equation*}
   \begin{split}
       w_0^{(q)} &\leq   \int_{-k}^{0} \beta_q(-\zeta) \left|h(\zeta/k) \right| d\zeta
        \leq  \int_{0}^{k} \beta_q(s) ds \leq \|\beta_q\|_{\mathrm{L_1}(0,\infty)}.
   \end{split}
   \end{equation*}
Therefore, taking $s_0=\frac{c}{2}$, then \eqref{eq4.13} becomes
   \begin{equation*}
   \begin{split}
        \left\|\widetilde{U}\left(e^{\frac{c}{2}+ \mathrm{i} \eta}\right)\right\|_{\widetilde{\mathrm{H}}}
         & \leq  \frac{e^{\frac{c}{2}}(1+e^{\frac{c}{2}})^2}{2(e^{\frac{c}{2}}-1)^3} \Big( k\|f^0\| + 2\|U^0\| \Big) + \frac{k(1+e^{\frac{c}{2}})^2}{2(e^{c}-1)} \sum\limits_{n=0}^{\infty} e^{-\frac{c}{2} n} \|f^n\| \\
        & +   \sum\limits_{q=1}^{2} \frac{(1+e^{\frac{c}{2}})^2  \|\beta_q\|_{\mathrm{L_1}(0,\infty)}}{2e^{\frac{c}{2}}(e^{\frac{c}{2}}-1)} k\|\mathbf{B}_qU^0\|.
   \end{split}
   \end{equation*}
Further, it holds that
   \begin{equation}\label{eq4.15}
   \begin{split}
        \left\|\widetilde{U}\left(e^{\frac{c}{2}+ \mathrm{i} \eta}\right)\right\|^2_{\widetilde{\mathrm{H}}}
         & \leq  \frac{e^{c}(1+e^{\frac{c}{2}})^4}{(e^{\frac{c}{2}}-1)^6} \Big( k\|f^0\| + 2\|U^0\|  \Big)^2 \\
        & +  \sum\limits_{q=1}^{2} \frac{ (1+e^{\frac{c}{2}})^4\|\beta_q\|^2_{\mathrm{L_1}(0,\infty)}}{e^{c}(e^{\frac{c}{2}}-1)^2} k^2 \|\mathbf{B}_qU^0\|^2  \\
        & +  \frac{(1+e^{\frac{c}{2}})^4}{2(e^{c}-1)^2} k^2 \left(\sum\limits_{n=0}^{\infty} e^{-\frac{c}{2} n} \|f^n\|\right)^2.
   \end{split}
   \end{equation}
Using \eqref{eq4.14} and \eqref{eq4.15}, the proof is finished.
\end{proof}

\vskip 2mm
Then we shall consider the convergence of the scheme \eqref{eq4.8}, and the following theorem holds.
\begin{theorem}\label{theorem4.3} Let $U^n$ and $u(t_n)$ be the solution of \eqref{eq4.8} and \eqref{eq1.3}, respectively. Supposing that $\beta_q(t)\in \mathrm{L_1}(\mathbb{R}_+)$, $q=1,2$ and $c>0$, then for $n\geq 1$,
   \begin{equation*}
   \begin{split}
      \|U^n-u(t_n)\|_A
         & \leq  \mathcal{C}k  \Biggr\{  e^{-\frac{c}{2}} \int_0^{k} \left\|u_{tt}(\zeta)\right\| d\zeta + k \sum\limits_{n=2}^{\infty}e^{-\frac{c}{2}n} \int_{t_{n-1}}^{t_{n}} \left\|u_{ttt}(\zeta)\right\| d\zeta  \\
         & +  \sum\limits_{q=1}^{2}\|\beta_q\|_{\mathrm{L_1}(0,\infty)} \sum\limits_{n=1}^{\infty} e^{-\frac{c}{2}n} \left(\int_{0}^{k}\left\|\mathbf{B}_q u_{t}(\zeta)\right\| d\zeta + k \int_{k}^{t_n}\left\|\mathbf{B}_qu_{tt}(\zeta)\right\| d\zeta \right)   \Biggr\}.
   \end{split}
   \end{equation*}
\end{theorem}
\begin{proof}
  It is obvious that temporal error $\rho^n = U^n - u(t_n)$ satisfies
   \begin{equation*}
   \begin{split}
      \delta_t \rho^n  & + \sum\limits_{q=1}^{2}Q_{n-1/2}^{(q)}(\mathbf{B}_q\rho) = r_1^{n-\frac{1}{2}} + r_2^{n-\frac{1}{2}},   \quad   n\geq 1, \\
      & \rho^0=0,
   \end{split}
   \end{equation*}
from which,
   \begin{equation}\label{eq4.16}
   \begin{split}
          r_1^{n-\frac{1}{2}} = \left(\frac{\partial u}{\partial t}\right)^{n-\frac{1}{2}} - \left(\frac{\partial u}{\partial t}\right)(t_{n-\frac{1}{2}}) + \left(\frac{\partial u}{\partial t}\right)(t_{n-\frac{1}{2}}) - \delta_tu^n.
   \end{split}
   \end{equation}
By denoting the notation $u_{t}=\frac{\partial u}{\partial t}$, we further get
   \begin{equation}\label{eq4.17}
   \begin{split}
       \left\|r_1^{\frac{1}{2}}\right\| \leq \int_0^{k} \left\|u_{tt}(\zeta)\right\| d\zeta,
       \quad \left\|r_1^{n-\frac{1}{2}}\right\|\leq \frac{k}{2} \int_{t_{n-1}}^{t_{n}} \left\|u_{ttt}(\zeta)\right\| d\zeta,
       \quad n\geq 2.
   \end{split}
   \end{equation}
Furthermore, it follows from \eqref{eq4.5} that
   \begin{equation}\label{eq4.18}
   \begin{split}
         \left\|r_2^n \right\|  \leq \sum\limits_{q=1}^{2} \left(  2  \breve{\mu}_{n-1}^{(q)} \int_{0}^{k}\left\|\mathbf{B}_qu_{t}(\zeta)\right\|d\zeta + k\sum\limits_{j=2}^{n} \breve{\mu}_{n-j}^{(q)} \int_{t_{j-1}}^{t_{j}} \left\|\mathbf{B}_qu_{tt}(\zeta)\right\| d\zeta \right).
   \end{split}
   \end{equation}
Using \eqref{eq4.18} and the definition of $\breve{\mu}_{j}^{(q)}$, we yield the following estimate
   \begin{equation}\label{eq4.19}
   \begin{split}
       \sum\limits_{n=1}^{\infty} e^{-\frac{c}{2}n}\left\|r_2^n \right\| \leq  \sum\limits_{q=1}^{2}\|\beta_q\|_{\mathrm{L_1}(0,\infty)} \sum\limits_{n=1}^{\infty} e^{-\frac{c}{2}n} \left( 2\int_{0}^{k}\left\|\mathbf{B}_qu_{t}(\zeta)\right\| d\zeta + k \int_{k}^{t_n}\left\|\mathbf{B}_qu_{tt}(\zeta)\right\| d\zeta \right).
   \end{split}
   \end{equation}
Then, from Theorem \ref{theorem4.2}, we have
   \begin{equation}\label{eq4.20}
   \begin{split}
       \sum\limits_{n=1}^{\infty} e^{-cn}\|\rho^n\|^2
         & \leq  \frac{\mathcal{C}(1+e^{\frac{c}{2}})^4}{(e^{c}-1)^2} k^2 \left[\sum\limits_{n=1}^{\infty} e^{-\frac{c}{2}n} \left( \left\|r_1^{n-\frac{1}{2}}\right\| + \left\|r_2^{n-\frac{1}{2}}\right\| \right) \right]^2.
   \end{split}
   \end{equation}
Substituting \eqref{eq4.17} and \eqref{eq4.18} into \eqref{eq4.20}, we complete the proof.
\end{proof}

\section{BDF2 IQ-CQ method for kernels of Case II}\label{section5}
Herein, we consider a mixed case with the kernels of \textbf{Case II}, namely, $\beta_1(t)\in \mathrm{L_{1}}(0,\infty)$ and $\beta_2(t)=\frac{t^{\alpha-1}}{\Gamma(\alpha)} \in \mathrm{L_{1,loc}}(0,\infty)$.

Firstly, in Section \ref{section4}, we have introduce the IQ rule $Q_{n}^{(1)}(\varphi)$ to approximate the first integral term in \eqref{eq1.3} with $\beta_1(t)\in \mathrm{L_{1}}(0,\infty)$, which means that we only need to consider another method to approximate the second integral term in \eqref{eq1.3} with $\beta_2(t)=\frac{t^{\alpha-1}}{\Gamma(\alpha)}$. Naturally, a better option is the second-order CQ rule \cite{Lubich,Lubich1}. 

Denote
     \begin{equation}\label{eq5.1}
     \begin{split}
             \widetilde{Q}_{n}^{(q)}(\varphi) = \varpi_{n,0}^{(q)} \varphi^0 + k^{\alpha}\sum\limits_{p=1}^{n} \omega_{n-p}^{(q)} \varphi^{p}, \quad q=1,2,
     \end{split}
     \end{equation}
from which, when $q=1$, $\widetilde{Q}_{n}^{(1)}(\varphi)= Q_{n}^{(1)}(\varphi)$, which implies that $\varpi_{n,0}^{(1)}=\chi_{n,0}^{(1)}$ and $k^{\alpha}\omega_{n}^{(1)}=w_{n}^{(1)}$; when $q=2$, $\widetilde{Q}_{n}^{(2)}(\varphi)$ is the second-order CQ rule, thus the generating coefficients $\omega_{n}^{(2)}$ is obtained by
     \begin{equation}\label{eq5.2}
     \begin{split}
             \hat{\beta}_2\left( \frac{(3-\zeta)(1-\zeta)}{2} \right) = \left[ \frac{(3-\zeta)(1-\zeta)}{2} \right]^{-\alpha} = \sum\limits_{n=0}^{\infty} \omega_{n}^{(2)}\zeta^{n},
     \end{split}
     \end{equation}
then in order to approximate the integral formally to the second order, we give the starting quadrature weights $\varpi_{n,0}^{(2)}$, such that
     \begin{equation}\label{eq5.3}
     \begin{split}
             \varpi_{n,0}^{(2)}  = ( \beta_2 \ast 1 )(t_n) -  k^{\alpha}\sum\limits_{p=1}^{n} \omega_{p}^{(2)},  \quad n\geq 1.
     \end{split}
     \end{equation}

In addition, we utilize the Taylor expansion formula and \cite[(2.4) and (3.3)]{Lubich} to obtain
   \begin{equation}\label{eq5.4}
   \begin{split}
       \left\| (\beta_2 \ast \varphi)(t_n)  - \widetilde{Q}_{n}^{(2)}(\varphi) \right\|  \leq \mathcal{C} & \Biggr(  k^2t_n^{\alpha-1} \left\|\varphi_{t}(0)\right\|  + k^{\alpha+1}\int_{t_{n-1}}^{t_n} \left\|\varphi_{tt}(\zeta)\right\|d\zeta  \\
       &+ k^{2}\int_{0}^{t_{n-1}} \beta_2(t_n-\zeta) \left\|\varphi_{tt}(\zeta)\right\| d\zeta  \Biggr).
   \end{split}
   \end{equation}

Next, the BDF2 method and \eqref{eq5.1} are used to approximate \eqref{eq1.3}, then
   \begin{equation}\label{eq5.5}
   \begin{split}
      \delta_t U^1  & + \sum\limits_{q=1}^{2}\widetilde{Q}_{1}^{(q)}(\mathbf{B}_qU)   = f^{1},   \\
      \delta_t^{(2)} U^n  & + \sum\limits_{q=1}^{2} \widetilde{Q}_{n}^{(q)}(\mathbf{B}_qU)   = f^{n},   \quad   n\geq 2, \\
      & U^0=u_0,
   \end{split}
   \end{equation}
where $\delta_t v^n$ defined by \eqref{eq4.8} and $\delta_t^{(2)} v^n=\frac{3}{2}\delta_t v^n - \frac{1}{2}\delta_t v^{n-1}$ with $n\geq 2$.

\vskip 1mm
Based on Propositions \ref{proposition3.1} and \ref{proposition3.2}, using the z-transform for \eqref{eq5.5} with $n\geq 1$, i.e., multiplying \eqref{eq5.5} with $z^{-n}$ and then summing it for $n$ from $1$ to $\infty$, we have
   \begin{equation}\label{eq5.6}
   \begin{split}
        \left( \frac{3}{2} - 2z^{-1} + \frac{z^{-2}}{2} \right) \vec{U}(z) & +  k \sum\limits_{q=1}^{2} \vec{\omega}^{(q)}(z)\mathbf{B}_q\vec{U}(z)
        = k\vec{F}(z)  \\
       &- k\sum\limits_{q=1}^{2} \vec{\varpi}^{(q)}(z) \mathbf{B}_qU^0 + \frac{z^{-1}}{2}U^1 + \left( z^{-1} - \frac{z^{-2}}{2} \right)U^0,
   \end{split}
   \end{equation}
from which,
   \begin{equation}\label{eq5.7}
   \begin{split}
       & \vec{U}(z) = U^1 z^{-1} + U^2 z^{-2} + \cdots + U^n z^{-n} + \cdots, \\
       & \vec{F}(z) = f^1 z^{-1} + f^2 z^{-2} + \cdots + f^n z^{-n} + \cdots, \\
       & \vec{\omega}^{(1)}(z) = \widetilde{w}^{(1)}(z), \quad  \vec{\varpi}^{(1)}(z) = \chi_{1,0}^{(1)}z^{-1} +\chi_{2,0}^{(1)}z^{-2} + \cdots + \chi_{n,0}^{(1)} z^{-n} + \cdots, \\
       & \vec{\omega}^{(2)}(z) = \omega_0^{(2)} + \omega_1^{(2)} z^{-1} + \cdots + \omega_n^{(2)} z^{-n} + \cdots = \left[ \frac{(3-z^{-1})(1-z^{-1})}{2} \right]^{-\alpha},  \\
       & \vec{\varpi}^{(2)}(z) = \varpi_{1,0}^{(2)}z^{-1} + \varpi_{2,0}^{(2)} z^{-2} + \cdots + \varpi_{n,0}^{(2)} z^{-n} + \cdots.
   \end{split}
   \end{equation}

\vskip 0.2mm
Let $z=e^{s_0+\mathrm{i}\eta}$ with $s_0>0$. Thus, we have
   \begin{equation}\label{eq5.8}
   \begin{split}
        \Re \left( \frac{3}{2} - 2z^{-1} + \frac{z^{-2}}{2} \right) = \frac{1-e^{-2s_0}}{2} + (1-e^{-s_0}\cos\eta)^2 \geq \frac{1-e^{-2s_0}}{2},
   \end{split}
   \end{equation}
   \begin{equation}\label{eq5.9}
   \begin{split}
        \Big|\Re \left( \vec{\varpi}^{(1)}(z) \right) \Big|  & \leq \sum\limits_{n=1}^{\infty} \left|\chi_{n,0}^{(1)}\right| e^{-s_0n} \leq 2\sum\limits_{n=1}^{\infty}
        \int_{t_{n-1}}^{t_{n}} e^{-s_0n} \beta_1(s) ds \\
        & \leq 2e^{-s_0} \int_{0}^{\infty}  \beta_1(s)ds \leq 2e^{-s_0} \|\beta_1\|_{\mathrm{L_1}(0,\infty)},
   \end{split}
   \end{equation}
and from \cite{Lubich1} we have $\left|\varpi_{n,0}^{(2)}\right|=\mathcal{O}(kt_n^{\alpha-1})$, which implies that
   \begin{equation}\label{eq5.10}
   \begin{split}
        \Big|\Re \left( \vec{\varpi}^{(2)}(z) \right) \Big|  &  \leq \sum\limits_{n=1}^{\infty} \left|\varpi_{n,0}^{(2)}\right| e^{-s_0n} \leq \mathcal{C}k^{\alpha} \sum\limits_{n=1}^{\infty} n^{\alpha-1} e^{-s_0n} \leq \mathcal{C}_{\alpha,s_0} k^{\alpha}.
   \end{split}
   \end{equation}

\vskip 0.5mm
In addition, a key lemma is established as follows.
\begin{lemma}\label{lemma5.1} Let the series $\vec{\omega}^{(2)}(z)=\sum_{n=0}^{\infty}\omega_n^{(2)}z^{-n}$ be given in \eqref{eq5.7}. Then it holds that $\Re \left(\vec{\omega}^{(2)}(z)\right) \geq 0$.
\end{lemma}
\begin{proof} Taking $z=e^{s_0+\mathrm{i}\eta}$ with $s_0>0$ and using \eqref{eq5.8}, we have $\Re \left( \frac{3}{2} - 2z^{-1} + \frac{z^{-2}}{2} \right)>0$. Therefore, the desired result is yielded by Lemma \ref{lemma2.2}.
\end{proof}

\vskip 2mm
Then, we shall deduce the following stability result.
\begin{theorem}\label{theorem5.2} Let $U^n$ denoted by \eqref{eq5.5} be the numerical solution of \eqref{eq1.3}. Supposing that $\beta_1(t)\in \mathrm{L_1}(\mathbb{R}_+)$, $\beta_2(t)=\frac{t^{\alpha-1}}{\Gamma(\alpha)} \in \mathrm{L_{1,loc}}(\mathbb{R}_+)$ and $c>0$, then it holds that
   \begin{equation*}
   \begin{split}
       \|U^n\|_A
        & \leq \frac{4k}{1-e^{-c}}  \sum\limits_{n=1}^{\infty} e^{-\frac{c}{2}n } \|f^n\|
        +  \frac{4e^{-\frac{c}{2}}}{1-e^{-c}}k \|\beta_1\|_{\mathrm{L_1}(0,\infty)}\left\|\mathbf{B}_1U^0\right\|  \\
        & + \frac{2\mathcal{C}_{\alpha,c}}{1-e^{-c}} k^{\alpha+1}\left\|\mathbf{B}_2U^0\right\|  + \frac{1+3e^{-\frac{c}{2}}}{1-e^{-c}}\left\|U^0\right\|.
   \end{split}
   \end{equation*}
\end{theorem}
\begin{proof} First, by taking the inner product of \eqref{eq5.6} with $\vec{U}(z)$, we have
   \begin{equation}\label{eq5.11}
   \begin{split}
        \left( \frac{3}{2} - 2z^{-1} + \frac{z^{-2}}{2} \right)& \left\|\vec{U}(z)\right\|^2_{\widetilde{\mathrm{H}}}  +   k \sum\limits_{q=1}^{2} \vec{\omega}^{(q)}(z)\left( \mathbf{B}_q\vec{U}(z), \vec{U}(z)\right)\\
        &\leq  k\left(\vec{F}(z), \vec{U}(z)\right)
       - k\sum\limits_{q=1}^{2} \vec{\varpi}^{(q)}(z) \left(\mathbf{B}_qU^0,  \vec{U}(z)\right) \\
       & + \frac{z^{-1}}{2}\left(U^1, \vec{U}(z)\right) + \left( z^{-1} - \frac{z^{-2}}{2} \right)\left(U^0, \vec{U}(z)\right),
   \end{split}
   \end{equation}
in which, employing the positivity of the operator $\mathbf{B}_q$, Lemma \ref{lemma4.1} and Lemma \ref{lemma5.1}, we obtain
   \begin{equation}\label{eq5.12}
   \begin{split}
        \sum\limits_{q=1}^{2} \Re\left(\vec{\omega}^{(q)}(z) \right) \left( \mathbf{B}_q\vec{U}(z), \vec{U}(z)\right) \geq 0.
   \end{split}
   \end{equation}
Then taking the real part of \eqref{eq5.11} and using \eqref{eq5.8} and \eqref{eq5.12}, we further get
   \begin{equation}\label{eq5.13}
   \begin{split}
        \frac{1-e^{-2s_0}}{2} \left\|\vec{U}(z)\right\|_{\widetilde{\mathrm{H}}} & \leq  k \sum\limits_{n=1}^{\infty} e^{-n s_0 } \|f^n\|
        + k\sum\limits_{q=1}^{2} \left|\Re\left(\vec{\varpi}^{(q)}(z)\right)\right|  \left\|\mathbf{B}_qU^0\right\| \\
        & + \frac{1}{2}e^{-s_0} \left\|U^1\right\| + \left(e^{-s_0}+\frac{1}{2}e^{-2s_0}\right)\left\|U^0\right\|.
   \end{split}
   \end{equation}
Noting that $\vec{U}(z)=\widetilde{U}(z)-U^0$ and using \eqref{eq5.9}-\eqref{eq5.10}, then \eqref{eq5.13} is simplified to
   \begin{equation}\label{eq5.14}
   \begin{split}
       \left\|\widetilde{U}(z)\right\|_{\widetilde{\mathrm{H}}} & \leq \frac{2k}{1-e^{-2s_0}}  \sum\limits_{n=1}^{\infty} e^{-n s_0 } \|f^n\|
        +  \frac{4e^{-s_0}}{1-e^{-2s_0}}k \|\beta_1\|_{\mathrm{L_1}(0,\infty)}\left\|\mathbf{B}_1U^0\right\|  \\
        & + \frac{2\mathcal{C}_{\alpha,s_0}}{1-e^{-2s_0}} k^{\alpha+1}\left\|\mathbf{B}_2U^0\right\|  + \frac{1+2e^{-s_0}}{1-e^{-2s_0}}\left\|U^0\right\| + \frac{e^{-s_0} }{1-e^{-2s_0}}\left\|U^1\right\|,
   \end{split}
   \end{equation}
from which, it follows from \eqref{eq5.5} that
   \begin{equation}\label{eq5.15}
   \begin{split}
       \left\|U^1\right\| \leq  \left\|U^0\right\| + 2k\left\|f^1\right\|.
   \end{split}
   \end{equation}
Then taking $s_0=\frac{c}{2}$, \eqref{eq5.14} and \eqref{eq5.15} lead to
   \begin{equation*}
   \begin{split}
        \left\|\widetilde{U}\left(e^{\frac{c}{2}+ \mathrm{i} \eta}\right)\right\|_{\widetilde{\mathrm{H}}}
         & \leq \frac{4k}{1-e^{-c}}  \sum\limits_{n=1}^{\infty} e^{-\frac{c}{2}n } \|f^n\|
        +  \frac{4e^{-\frac{c}{2}}}{1-e^{-c}}k \|\beta_1\|_{\mathrm{L_1}(0,\infty)}\left\|\mathbf{B}_1U^0\right\|  \\
        & + \frac{2\mathcal{C}_{\alpha,c}}{1-e^{-c}} k^{\alpha+1}\left\|\mathbf{B}_2U^0\right\|  + \frac{1+3e^{-\frac{c}{2}}}{1-e^{-c}}\left\|U^0\right\|.
   \end{split}
   \end{equation*}
Then combining above formula and \eqref{eq4.14}, we complete the proof.
\end{proof}

\vskip 2mm
Then the convergence of the scheme \eqref{eq5.5} will be deduced as follows.
\begin{theorem}\label{theorem5.3} Let $U^n$ and $u(t_n)$ be the solution of \eqref{eq5.5} and \eqref{eq1.3}, respectively. Assuming that $\beta_1(t)\in \mathrm{L_1}(\mathbb{R}_+)$, $\beta_2(t)=\frac{t^{\alpha-1}}{\Gamma(\alpha)} \in \mathrm{L_{1,loc}}(\mathbb{R}_+)$ and $c>0$, then for $n\geq 1$, it holds that
   \begin{equation*}
   \begin{split}
      \|U^n-u(t_n)\|_A
         & \leq  \frac{4k}{1-e^{-c}} \Biggr[   \frac{8}{c} \int_{0}^{2k} \|u_{tt}(\zeta)\|d\zeta + k\sum\limits_{n=2}^{\infty} e^{-\frac{c}{2}n }\int_{t_{n-1}}^{t_{n}} \|u_{ttt}(\zeta)\|d\zeta \\
         & + \|\beta_1\|_{\mathrm{L_1}(0,\infty)} \sum\limits_{n=1}^{\infty} e^{-\frac{c}{2}n}  \Biggr( 2\int_{0}^{k}\left\|\mathbf{B}_1u_{t}(\zeta)\right\| d\zeta +  k \int_{k}^{t_n}\left\|\mathbf{B}_1u_{tt}(\zeta)\right\| d\zeta \Biggr) \\
         & \qquad +  \sum\limits_{n=1}^{\infty} e^{-\frac{c}{2}n}   \Biggr(  k^2t_n^{\alpha-1} \left\|\mathbf{B}_2u_{t}(0)\right\| + k^{\alpha+1}\int_{t_{n-1}}^{t_n} \left\|\mathbf{B}_2u_{tt}(\zeta)\right\| d\zeta  \\
        & \qquad\qquad\qquad + k^{2}\int_{0}^{t_{n-1}} (t_n-\zeta)^{\alpha-1} \left\|\mathbf{B}_2u_{tt}(\zeta)\right\| d\zeta  \Biggr)
        \Biggr].
   \end{split}
   \end{equation*}
\end{theorem}
\begin{proof}
  At first, the time error $\rho^n = U^n - u(t_n)$ satisfies
   \begin{equation}\label{eq5.16}
   \begin{split}
     \delta_t \rho^1  & + \sum\limits_{q=1}^{2}\widetilde{Q}_{1}^{(q)}(\mathbf{B}_q\rho)   = r_3^{1} + r_4^1 + r_5^1,   \\
     \delta_t^{(2)} \rho^n  & + \sum\limits_{q=1}^{2} \widetilde{Q}_{n}^{(q)}(\mathbf{B}_q\rho) = r_3^{n} + r_4^{n} + r_5^{n},   \quad   n\geq 2, \\
      & \rho^0=0,
   \end{split}
   \end{equation}
where the notations $r_3^{n}$, $n\geq 1$ are the truncation errors of first-order time derivative, and the following estimates hold
   \begin{equation}\label{eq5.17}
   \begin{split}
         & \|r_3^{1}\| \leq  2 \int_{0}^{k} \|u_{tt}(\zeta)\|d\zeta, \\
         & \|r_3^{n}\| \leq 4 \left( \int_{0}^{2k} \|u_{tt}(\zeta)\|d\zeta + k\int_{t_{n-1}}^{t_{n}} \|u_{ttt}(\zeta)\|d\zeta  \right),  \quad  n\geq 2.
   \end{split}
   \end{equation}
Furthermore, $r_4^{n}$ and $r_5^{n}$ are the quadrature errors of two integrals in \eqref{eq1.3}, which are estimated by \eqref{eq4.5} and \eqref{eq5.4}, respectively. Next, regarding \eqref{eq5.16}, similar to the analysis of Theorem \ref{theorem5.2}, we have
   \begin{equation}\label{eq5.18}
   \begin{split}
       \|\rho^n\|_A
        & \leq \frac{4k}{1-e^{-c}}  \sum\limits_{n=1}^{\infty} e^{-\frac{c}{2}n } \big(\|r_3^n\| +\|r_4^n\| +\|r_5^n\| \big),
   \end{split}
   \end{equation}
from which, \eqref{eq5.7} leads to
   \begin{equation}\label{eq5.19}
   \begin{split}
       \sum\limits_{n=1}^{\infty} e^{-\frac{c}{2}n } \|r_3^n\| \leq    \frac{8}{c} \int_{0}^{2k} \|u_{tt}(\zeta)\|d\zeta + k\sum\limits_{n=2}^{\infty} e^{-\frac{c}{2}n }\int_{t_{n-1}}^{t_{n}} \|u_{ttt}(\zeta)\|d\zeta,
   \end{split}
   \end{equation}
and using \eqref{eq4.19}, we obtain
   \begin{equation}\label{eq5.20}
   \begin{split}
       \sum\limits_{n=1}^{\infty} e^{-\frac{c}{2}n}\left\|r_4^n \right\| \leq  \|\beta_1\|_{\mathrm{L_1}(0,\infty)} \sum\limits_{n=1}^{\infty} e^{-\frac{c}{2}n}  \Biggr( 2\int_{0}^{k}\left\|\mathbf{B}_1u_{t}(\zeta)\right\| d\zeta +  k \int_{k}^{t_n}\left\|\mathbf{B}_1u_{tt}(\zeta)\right\| d\zeta \Biggr).
   \end{split}
   \end{equation}
Then employing \eqref{eq5.4}, we have
   \begin{equation}\label{eq5.21}
   \begin{split}
     \sum\limits_{n=1}^{\infty} e^{-\frac{c}{2}n}\left\|r_5^n \right\| \leq  \sum\limits_{n=1}^{\infty} e^{-\frac{c}{2}n}  & \Biggr(  k^2t_n^{\alpha-1} \left\|\mathbf{B}_2u_{t}(0)\right\| + k^{\alpha+1}\int_{t_{n-1}}^{t_n} \left\|\mathbf{B}_2u_{tt}(\zeta)\right\| d\zeta  \\
       &+ k^{2}\int_{0}^{t_{n-1}} (t_n-\zeta)^{\alpha-1} \left\|\mathbf{B}_2u_{tt}(\zeta)\right\| d\zeta  \Biggr).
   \end{split}
   \end{equation}
By substituting \eqref{eq5.19}-\eqref{eq5.21} into \eqref{eq5.18}, the proof is finished.
\end{proof}

\section{Crank-Nicolson TCQ method for kernels of Case III}\label{section6}

In this section, we consider the kernels of \textbf{Case III}, that is $\beta_q(t)=\frac{t^{\alpha_q-1}}{\Gamma(\alpha_q)} \in \mathrm{L_{1,loc}}(0,\infty)$, $q=1,2$.

Firstly, we apply the trapezoidal convolution quadrature rule; see \cite{Cuesta1,Lubich}, i.e.,
     \begin{equation}\label{eq6.1}
     \begin{split}
             \mathcal{Q}_{n}^{(q)}(\varphi) = \kappa_{n,0}^{(q)} \varphi^0 + k^{\alpha_q}\sum\limits_{p=0}^{n} \mu_{n-p}^{(q)} \varphi^{p}, \quad q=1,2,
     \end{split}
     \end{equation}
from which, the generating coefficients $\mu_{n}^{(q)}$ of the power series are given by
     \begin{equation}\label{eq6.2}
     \begin{split}
             \hat{\beta_q}\left( \frac{2(1-z)}{1+z} \right) = \left( \frac{2(1-z)}{1+z} \right)^{-\alpha_q} = \sum\limits_{m=0}^{\infty} \mu_{m}^{(q)}z^{m},
     \end{split}
     \end{equation}
and for approximating the integral term formally to order 2, we introduce the starting quadrature weights $\kappa_{n,0}^{(q)}$, satisfying
     \begin{equation}\label{eq6.3}
     \begin{split}
             \kappa_{n,0}^{(q)}  = ( \beta_q \ast 1 )(t_n) -  k^{\alpha_q}\sum\limits_{p=0}^{n} \mu_{n}^{(q)},  \quad n\geq 0.
     \end{split}
     \end{equation}
Then, denoting $\widetilde{\kappa}_n^{(q)}=\frac{1}{2}\left(\kappa_{n,0}^{(q)}+\kappa_{n-1,0}^{(q)} + k^{\alpha_q}\mu_n^{(q)} \right)$ for $n\geq 1$, we have
     \begin{equation}\label{eq6.4}
     \begin{split}
             \mathcal{Q}_{n-\frac{1}{2}}^{(q)}(\varphi) = \widetilde{\kappa}_n^{(q)} \varphi^0 + k^{\alpha_q}\sum\limits_{p=1}^{n} \mu_{n-p}^{(q)} \varphi^{p-\frac{1}{2}}.
     \end{split}
     \end{equation}

Below, the Crank-Nicolson method and \eqref{eq6.4} are applied to discretize \eqref{eq1.3}, then
   \begin{equation}\label{eq6.5}
   \begin{split}
      \delta_t U^n  & + \sum\limits_{q=1}^{2}\mathcal{Q}_{n-1/2}^{(q)}(\mathbf{B}_qU)   = f^{n-1/2}, \quad   n\geq 1,  \\
      & U^0=u_0.
   \end{split}
   \end{equation}

\vskip 0.2mm
Next, we apply the z-transform to second and third terms of left side of \eqref{eq6.5}, then
     \begin{equation}\label{eq6.6}
     \begin{split}
            \sum\limits_{n=1}^{\infty} \mathcal{Q}_{n-\frac{1}{2}}^{(q)}(\mathbf{B}_qU) z^{-n}& = \frac{(1+z^{-1})k^{\alpha_q}}{2} \widetilde{\mu}^{(q)}(z) \mathbf{B}_q\widetilde{U}(z) \\
            & - \left(\frac{k^{\alpha_q} }{2} \widetilde{\mu}^{(q)}(z) - \widetilde{\kappa}^{(q)}(z) \right) \mathbf{B}_qU^0, \quad q=1,2,
     \end{split}
     \end{equation}
where $\widetilde{U}(z)$ defined by \eqref{eq4.11}, and
   \begin{equation}\label{eq6.7}
   \begin{split}
         & \widetilde{\mu}^{(q)}(z) = \mu_0^{(q)} + \mu_1^{(q)} z^{-1} + \cdots + \mu_n^{(q)} z^{-n} + \cdots = \left( \frac{2(1-z^{-1})}{1+z^{-1}} \right)^{-\alpha_q},   \\
         & \widetilde{\kappa}^{(q)}(z) = \widetilde{\kappa}_1^{(q)}z^{-1} + \widetilde{\kappa}_2^{(q)} z^{-2} + \cdots + \widetilde{\kappa}_n^{(q)} z^{-n} + \cdots, \qquad q=1,2.
   \end{split}
   \end{equation}
Then, by calculation and simplification, we get
   \begin{equation}\label{eq6.8}
   \begin{split}
         \Theta^{(q)}(z) \triangleq \frac{2}{1+z^{-1}} \left( \frac{k^{\alpha_q}}{2}\widetilde{\mu}^{(q)}(z) - \widetilde{ \kappa}^{(q)}(z) \right) = \frac{ k^{\alpha_q}\mu_0^{(q)}- \kappa_{0,0}^{(q)}}{1+z^{-1} } - \kappa^{(q)}(z),
   \end{split}
   \end{equation}
in which $\kappa^{(q)}(z) = \kappa_{1,0}^{(q)}z^{-1} + \kappa_{2,0}^{(q)} z^{-2} + \cdots + \kappa_{n,0}^{(q)} z^{-n} + \cdots$. This naturally yields
   \begin{equation}\label{eq6.9}
   \begin{split}
      \Re \left(\Theta^{(q)}(z)\right)= \left(\frac{k}{2}\right)^{\alpha_q} \Re \left(\frac{1}{1+z^{-1}} \right) - \sum\limits_{n=1}^{\infty}\kappa_{n,0}^{(q)} e^{-ns_0}\cos(n\eta).
   \end{split}
   \end{equation}
Besides, using the estimate of $\kappa_{n,0}^{(q)}$ (see \cite[Theorem 3]{Cuesta1}), we have
   \begin{equation}\label{eq6.10}
   \begin{split}
      \left|- \sum\limits_{n=1}^{\infty}\kappa_{n,0}^{(q)} e^{-ns_0}\cos(n\eta)\right| &\leq   \mathcal{C}k^{\alpha_q}\sum\limits_{n=1}^{\infty}e^{-s_0n} n^{\alpha_q-1} \\
      & \leq \mathcal{C}_{\alpha_q}' k^{\alpha_q} \int_{0}^{\infty} e^{-s_0 \theta}d\theta  \leq \mathcal{C}_{\alpha_q, s_0}' k^{\alpha_q}.
   \end{split}
   \end{equation}

\vskip 0.5mm
Furthermore, we give a key lemma as follows.
\begin{lemma}\label{lemma6.1} Let the series $\widetilde{\mu}^{(q)}(z)=\sum\limits_{n=0}^{\infty}\mu_n^{(q)} z^{-n}$ be denoted in \eqref{eq6.7}. Then it holds that $\Re \left(\widetilde{\mu}^{(q)}(z)\right) \geq 0$, $q=1,2$.
\end{lemma}
\begin{proof} Let $s_0>0$ and $z=e^{s_0+\mathrm{i}\eta}$, then it holds that
$ \Re \left(\frac{2(1-z^{-1})}{1+z^{-1}} \right) \geq \frac{2(e^{2s_0}-1)}{(1+e^{s_0})^2} >0$, which follows from Lemma \ref{lemma2.2} that $\Re \left(\widetilde{\mu}^{(q)}(z)\right) \geq 0$.
\end{proof}

\vskip 2mm
Based on above analyses, we obtain the following stability result.
\begin{theorem}\label{theorem6.2} Let $U^n$ be denoted by \eqref{eq6.5} be the numerical solution of \eqref{eq1.3}. Supposing $\beta_q(t)=\frac{t^{\alpha_q-1}}{\Gamma(\alpha_q)}$, $q=1,2$ and $c>0$, it holds that
   \begin{equation*}
   \begin{split}
       \|U^n\|_A
         & \leq  \frac{e^{\frac{c}{2}}(1+e^{\frac{c}{2}})^2}{2(e^{\frac{c}{2}}-1)^3} \Big( k\|f^0\| + 2\|U^0\| \Big) \\
        & +   \sum\limits_{q=1}^{2} k^{1+\alpha_q} \mathcal{C}_{\alpha_q, c}'\|\mathbf{B}_qU^0\|  + \frac{(1+e^{\frac{c}{2}})^2}{2(e^{c}-1)} k\sum\limits_{n=0}^{\infty} e^{-\frac{c}{2} n} \|f^n\|.
   \end{split}
   \end{equation*}
\end{theorem}
\begin{proof}
  First, multiplying \eqref{eq6.5} by $z^{-n}$, and then summing for $n$ from 1 to $\infty$, we obtain
   \begin{equation}\label{eq6.11}
   \begin{split}
         \frac{2(1-z^{-1})}{1+z^{-1}} \widetilde{U}(z) & +  \sum\limits_{q=1}^{2}k^{1+\alpha_q}\widetilde{\mu}^{(q)}(z) \mathbf{B}_q\widetilde{U}(z)
       = k\widetilde{F}(z) - \frac{k}{1+z^{-1}} f^0 \\
       & + \frac{2}{1+z^{-1}} U^0  + k \sum\limits_{q=1}^{2} \Theta^{(q)}(z) \mathbf{B}_qU^0,
   \end{split}
   \end{equation}
where Propositions \ref{proposition3.1}-\ref{proposition3.2} and \eqref{eq6.6}-\eqref{eq6.8} are utilized, and $\widetilde{F}(z)$ defined by \eqref{eq4.11}. Further, taking the inner product of \eqref{eq6.11} with $\widetilde{U}(z)$, it holds that
   \begin{equation}\label{eq6.12}
   \begin{split}
         \frac{2(1-z^{-1})}{1+z^{-1}} &\left\|\widetilde{U}(z)\right\|^2_{\widetilde{\mathrm{H}}}  +  \sum\limits_{q=1}^{2}k^{1+\alpha_q}\widetilde{\mu}^{(q)}(z) \left(\mathbf{B}_q \widetilde{U}(z), \widetilde{U}(z)\right)
       \\
        & = k\left(\widetilde{F}(z),\widetilde{U}(z)\right)  - \frac{k}{1+z^{-1}} \left(f^0,\widetilde{U}(z)\right)   + \frac{2}{1+z^{-1}} \left(U^0,\widetilde{U}(z)\right) \\
        &   + k \sum\limits_{q=1}^{2} \Theta^{(q)}(z) \left(\mathbf{B}_qU^0,\widetilde{U}(z)\right),
   \end{split}
   \end{equation}
and then taking the real part of \eqref{eq6.12}, we employ \eqref{eq4.12}, \eqref{eq6.9}, \eqref{eq6.10} and Lemma \ref{lemma6.1} to yield
   \begin{equation}\label{eq6.13}
   \begin{split}
        \frac{2(e^{2s_0}-1)}{(1+e^{s_0})^2}  \left\|\widetilde{U}(e^{s_0+\mathrm{i} \eta})\right\|_{\widetilde{\mathrm{H}}}
         & \leq  \frac{1+e^{-s_0}}{(1-e^{-s_0})^2} \Big( k\|f^0\| + 2\|U^0\|  \Big) + k \sum\limits_{n=0}^{\infty} e^{-s_0 n} \|f^n\| \\
        & +    \sum\limits_{q=1}^{2}k^{1+\alpha_q} \left( \mathcal{C}_{\alpha_q, s_0}' +  \frac{1+e^{-s_0}}{2^{\alpha_q}(1-e^{-s_0})^2}   \right) \|\mathbf{B}_qU^0\|.
   \end{split}
   \end{equation}
Then choosing $s_0=\frac{c}{2}$, we yield
   \begin{equation}\label{eq6.14}
   \begin{split}
        \left\|\widetilde{U}\left(e^{\frac{c}{2}+ \mathrm{i} \eta}\right)\right\|_{\widetilde{\mathrm{H}}}
         & \leq  \frac{e^{\frac{c}{2}}(1+e^{\frac{c}{2}})^2}{2(e^{\frac{c}{2}}-1)^3} \Big( k\|f^0\| + 2\|U^0\| \Big)  + \frac{(1+e^{\frac{c}{2}})^2}{2(e^{c}-1)} k\sum\limits_{n=0}^{\infty} e^{-\frac{c}{2} n} \|f^n\|\\
        & +   \sum\limits_{q=1}^{2} k^{1+\alpha_q} \mathcal{C}_{\alpha_q, c}'\|\mathbf{B}_qU^0\|,
   \end{split}
   \end{equation}
where $\mathcal{C}_{\alpha_q, c}'= \frac{(1+e^{\frac{c}{2}})^2}{2(e^{c}-1)}\left( \mathcal{C}_{\alpha_q, \frac{c}{2}}' +  \frac{1+e^{-\frac{c}{2}}}{2^{\alpha_q}(1-e^{-\frac{c}{2}})^2}   \right)$. Then using \eqref{eq4.14} and \eqref{eq6.14}, we obtain the desired result.
\end{proof}

\vskip 2mm
Then we will consider the convergence of the scheme \eqref{eq6.5}, namely the following theorem.
\begin{theorem}\label{theorem6.3} Let $U^n$ and $u(t_n)$ be the solution of \eqref{eq6.5} and \eqref{eq1.3}, respectively. Assuming that $\beta_q(t)=\frac{t^{\alpha_q-1}}{\Gamma(\alpha_q)}$, $q=1,2$ and $c>0$, then for $n\geq 1$,
   \begin{equation*}
   \begin{split}
      \|U^n-u(t_n)\|_A
         & \leq  \mathcal{C}k  \Biggr\{  e^{-\frac{c}{2}} \int_0^{k} \left\|u_{tt}(\zeta)\right\| d\zeta + k \sum\limits_{n=2}^{\infty}e^{-\frac{c}{2}n} \int_{t_{n-1}}^{t_{n}} \left\|u_{ttt}(\zeta)\right\| d\zeta  \\
         & +  \sum\limits_{q=1}^{2}  \sum\limits_{n=1}^{\infty} e^{-\frac{c}{2}n}   \Biggr(  k^2t_n^{\alpha_q-1} \left\|\mathbf{B}_qu_{t}(0)\right\| + k^{\alpha_q+1}\int_{t_{n-1}}^{t_n} \left\|\mathbf{B}_qu_{tt}(\zeta)\right\| d\zeta  \\
        & \qquad\qquad + k^{2}\int_{0}^{t_{n-1}} (t_n-\zeta)^{\alpha_q-1} \left\|\mathbf{B}_qu_{tt}(\zeta)\right\| d\zeta  \Biggr) \Biggr\}.
   \end{split}
   \end{equation*}
\end{theorem}
\begin{proof}  First, by \eqref{eq6.5} we get the error equations as follows
\begin{equation}\label{eq6.15}
   \begin{split}
      \delta_t \rho^n  & + \sum\limits_{q=1}^{2}\mathcal{Q}_{n-1/2}^{(q)}(\mathbf{B}_q\rho) = R_1^{n-\frac{1}{2}} + R_2^{n-\frac{1}{2}},   \quad   n\geq 1, \\
      & \rho^0=0,
   \end{split}
   \end{equation}
where $R_1^{n-\frac{1}{2}}:=r_1^{n-\frac{1}{2}}$ is given and estimated by \eqref{eq4.16} and \eqref{eq4.17}, respectively, and
   \begin{equation}\label{eq6.16}
   \begin{split}
       R_2^{n-\frac{1}{2}} = \sum\limits_{q=1}^{2}\left[ (\beta_q \ast \mathbf{B}_qu)(t_n)  - \mathcal{Q}_{n}^{(q)}(\mathbf{B}_qu) \right],
   \end{split}
   \end{equation}
from which, we utilize the quadrature result of \cite[Theorem 2]{Cuesta1} to get
   \begin{equation}\label{eq6.17}
   \begin{split}
       \left\| (\beta_q \ast \varphi)(t_n)  - \mathcal{Q}_{n}^{(q)}(\varphi) \right\|  \leq \mathcal{C} & \Biggr(  k^2t_n^{\alpha_q-1} \left\|\varphi_{t}(0)\right\|  + k^{\alpha_q+1}\int_{t_{n-1}}^{t_n} \left\|\varphi_{tt}(\zeta)\right\|d\zeta  \\
       & \qquad + k^{2}\int_{0}^{t_{n-1}} \beta_q(t_n-\zeta) \left\|\varphi_{tt}(\zeta)\right\| d\zeta  \Biggr).
   \end{split}
   \end{equation}
Next, in view of Theorem \ref{theorem6.2} and \eqref{eq6.15}, we have
   \begin{equation}\label{eq6.18}
   \begin{split}
       \|\rho^n\|_A
         & \leq  \mathcal{C} k\sum\limits_{n=1}^{\infty} e^{-\frac{c}{2} n} \left( \left\|R_1^{n-\frac{1}{2}}  \right\| + \left\|R_2^{n-\frac{1}{2}} \right\| \right).
   \end{split}
   \end{equation}
By substituting \eqref{eq4.17}, \eqref{eq6.16} and \eqref{eq6.17} into \eqref{eq6.18}, the proof is completed.
\end{proof}

\vskip 2mm
Note that if the regularity of the solution of \eqref{eq1.3} can be given, then Theorem \ref{theorem6.3} will show a more concrete form, rather than an abstract representation. In fact, similar to that of \cite[Section 2, (2.1)]{McLean2}, when $u_0$ is relatively smooth, we assume that the solution of homogeneous case of \eqref{eq1.3} satisfies
   \begin{equation}\label{eq6.19}
   \begin{split}
          t^2\left\|u_{tt}(t)\right\| + t^3\left\|u_{ttt}(t)\right\| +  t\left\|\mathbf{B}_qu_t(t)\right\| + t^2 \left\|\mathbf{B}_qu_{tt}(t)\right\| \leq  \mathcal{C} t^{1+ \alpha_{*}}, \quad t\rightarrow 0^+,
   \end{split}
   \end{equation}
where $q=1,2$, and $\alpha_{*}=\min\left( \alpha_1,\alpha_2\right)$. Thence, we obtain the following theorem.

\begin{theorem}\label{theorem6.4} Let $U^n$ and $u(t_n)$ be the solution of \eqref{eq6.5} and \eqref{eq1.3}, respectively. Supposing that $\beta_q(t)=\frac{t^{\alpha_q-1}}{\Gamma(\alpha_q)}$, $q=1,2$, $\alpha_{*}=\min\left( \alpha_1,\alpha_2\right)$ and appropriate $c>0$, and satisfying the regularity assumption \eqref{eq6.19}, then it holds that
   \begin{equation*}
   \begin{split}
      \|U^n-u(t_n)\|_A  & \leq  \mathcal{C}k^{\alpha_*+1}.
   \end{split}
   \end{equation*}
\end{theorem}
\begin{proof} Using the assumption \eqref{eq6.19}, it is easy to yield
   \begin{equation}\label{eq6.20}
   \begin{split}
        &  k^2 \sum\limits_{n=2}^{\infty}e^{-\frac{c}{2}n} \int_{t_{n-1}}^{t_{n}} \left\|u_{ttt}(\zeta)\right\| d\zeta
        \leq \mathcal{C}  k^{1+\alpha_*} \sum\limits_{n=2}^{\infty}e^{-\frac{c}{2}n} n^{\alpha_*-1} \leq  \mathcal{C}k^{\alpha_*+1},
   \end{split}
   \end{equation}
and
   \begin{equation}\label{eq6.21}
   \begin{split}
        k^{3} \sum\limits_{n=1}^{\infty}e^{-\frac{c}{2}n}  & \int_{0}^{t_{n-1}} (t_n-\zeta)^{\alpha_q-1} \left\|\mathbf{B}_qu_{tt}(\zeta)\right\| d\zeta
        \\
        & \leq \mathcal{C}  k^{2+\alpha_*} \sum\limits_{n=1}^{\infty}e^{-\frac{c}{2}n} \int_{0}^{t_{n-1}} \zeta^{\alpha_* -1}d\zeta \leq  \mathcal{C}k^{\alpha_*+2}.
   \end{split}
   \end{equation}
The proof is finished by Theorem \ref{theorem6.3} and \eqref{eq6.20}-\eqref{eq6.21}.
\end{proof}

\section{Numerical experiment}\label{section7} In this section, we shall give some numerical examples to verify the theoretical analysis of Section \ref{section6}, involving the nonsmooth kernels of \textbf{Case III}, i.e., Abel kernels $\beta_q=\frac{t^{\alpha_q-1}}{\Gamma(\alpha_q)}$, $q=1,2$.

Below, we define the error and temporal convergence order with large $T$ (theoretically $T\rightarrow \infty$) and suitable $c>0$ that
   \begin{equation*}
   \begin{split}
                  Error(k,h) = \left( c \sum\limits_{n=1}^{\lfloor T/k \rfloor} e^{-c n} \|U^n-u(t_n)\|^2 \right)^{\frac{1}{2}}, \quad  Rate^t = \log_2\left( \frac{Error(2k,h)}{Error(k,h)}\right),
   \end{split}
   \end{equation*}
where the time-space step sizes $k=\frac{T}{N}$ and $h=\frac{1}{\mathcal{M}}$ with $\mathcal{M}, N\in \mathbb{Z}_+$.

\subsection{Application I}\label{subsection7.1} Concretely, from \cite{Hannsgen1}, we let $\mathbf{B}_1=-\frac{\partial^2}{\partial x^2}$ and $\mathbf{B}_2=\frac{\partial^4}{\partial x^4}$. Thus, with self-adjoint boundary and $\Omega=(0,1)$, problem \eqref{eq1.3} turns into
\begin{equation}\label{eq7.1}
	\begin{split}
           \frac{\partial u}{\partial t} & - \int_{0}^{t} \Big[\beta_1(t-s) u_{xx}(x,s) -  \beta_2(t-s) u_{xxxx}(x,s) \Big]ds  = f(x,t), \\
            & \qquad\qquad\qquad\qquad\qquad\qquad\qquad (x,t)\in \Omega\times (0,\infty),  \\
            &   u(x,0)  = u_0(x), \quad x\in \Omega \cup \partial\Omega, \\
            &   u(0,t)  = u(1,t) = u_{xx}(0,t)  = u_{xx}(1,t) = 0, \quad t \in (0,\infty).
    \end{split}
\end{equation}
Here, in order to test the long-time behavior of the scheme \eqref{eq6.5}, we approximate the terms $u_{xx}(x_j,t_n)$ and $u_{xxxx}(x_j,t_n)$ by the second-order difference method, that is
\begin{equation}\label{eq7.2}
	\begin{split}
             & \frac{\partial^2 u}{\partial x^2}(x_j,t_n) \approx \delta_x^2 U_j^n := \frac{U^n_{j+1}-2U^n_{j}+U^n_{j-1}}{h^2}, \quad   1\leq j \leq  \mathcal{M}-1, \quad   1\leq n \leq  N,  \\
             &  \frac{\partial^4 u}{\partial x^4}(x_j,t_n) \approx \delta_x^4U_j^n :=  \frac{U^n_{j+2}-4U^n_{j+1}+6U^n_{j}-4U^n_{j-1}+U^n_{j-2}}{h^4}.
    \end{split}
\end{equation}
Thence, we have the following fully discrete difference scheme
   \begin{equation}\label{eq7.3}
   \begin{split}
      \delta_t U^n_j  & - \mathcal{Q}_{n-\frac{1}{2}}^{(1)}(\delta_x^2U_j) + \mathcal{Q}_{n-\frac{1}{2}}^{(2)}(\delta_x^4U_j)   = f^{n-1/2}_j,
      \quad   n\geq 1, \quad   1\leq j \leq  \mathcal{M}-1,  \\
      & U_j^0=u_0(x_j), \quad 1\leq j \leq  \mathcal{M}-1, \\
      & U_{0}^n = U_{\mathcal{M}}^n = 0,  \quad  n\geq 1, \\
      & U_{-1}^n = -U_{1}^n, \quad U_{\mathcal{M}+1}^n = -U_{\mathcal{M}-1}^n,  \quad  n\geq 1.
   \end{split}
   \end{equation}

   \vskip 1mm
     \textbf{Example 1.} In this example, considering the problem \eqref{eq1.3} by the fully discrete scheme \eqref{eq7.3}. First, to satisfy the regularity \eqref{eq6.19} of the solution of \eqref{eq1.3}, we give the following exact solution
   \begin{equation}\label{eq7.4}
   \begin{split}
      u(x,t) = \sin \pi x - \frac{t^{\alpha_* +1}}{\Gamma(2+\alpha_*)} \sin 2\pi x,
   \end{split}
   \end{equation}
     then the initial condition $u_0(x)=\sin(\pi x)$ and the source term
   \begin{equation*}
   \begin{split}
      f(x,t) =& - \left(  \frac{t^{\alpha_*+1}}{\Gamma(\alpha_*+1)} + \frac{4\pi^2 t^{\alpha_*+\alpha_1+1}}{\Gamma(\alpha_*+\alpha_1+2)}  + \frac{16\pi^4 t^{\alpha_*+\alpha_2+1}}{\Gamma(\alpha_*+\alpha_2+2)}  \right)\sin 2\pi x \\
          &  + \left(  \frac{\pi^2 t^{\alpha_1}}{\Gamma(\alpha_1+1)} +  \frac{\pi^4 t^{\alpha_2}}{\Gamma(\alpha_2+1)}  \right)\sin \pi x.
   \end{split}
   \end{equation*}

     \vskip 0.2mm
     Table \ref{tb1} shows the errors and corresponding temporal convergence orders for fully discrete scheme \eqref{eq7.3}. Here, we test the long-time behavior of numerical solutions, i.e., the cases of large $T$. As we can see, when $T=100$ or $T=400$,
     the numerical results in Table \ref{tb1} are consistent with our theoretical estimate (see Theorem \ref{theorem6.4}). Concretely, the time convergence rates in Table \ref{tb1} approximate the order $k^{1+\alpha_*}$, which is in line with our expected results.

\begin{table}
  \center
    \footnotesize
  \caption{The errors and temporal convergence orders with different values of $\alpha_*$, $T$ and $\mathcal{M}=2048$.} \label{tb1}
  \begin{tabular}{cccccccc}
    \toprule
  & & \multicolumn{2}{c}{$T=100$, $c=1$}& & \multicolumn{2}{c}{$T=400$, $c=\frac{1}{4}$} \\
   \cmidrule(r){3-4}  \cmidrule(r){6-7}
   \noalign{\smallskip}
   $\alpha_*=\min(\alpha_1,\alpha_2)$ & $N$ & $Error(k,h)$ & $Rate^t$ & $N$ & $Error(k,h)$ & $Rate^t$ \\
    \midrule
                     & 8   &  5.1157e-01  &  -      & 16   & 9.6360e-01  & -       \\
     $\alpha_1=0.3$  & 16  &  2.0698e-01  &  1.305  & 32   & 3.9002e-01  & 1.305  \\
     $\alpha_2=0.7$  & 32  &  8.3697e-02  &  1.306  & 64   & 1.5775e-01  & 1.306  \\
                     & 64  &  3.3882e-02  &  1.305  & 128  & 6.3785e-02  & 1.306  \\
                     & 128 &  1.3820e-02  &  1.294  & 256  & 2.5864e-02  & 1.302  \\
    \midrule
                     & 8   &  1.5600e-01  &  -      & 16   & 3.6198e-01  & -       \\
     $\alpha_1=0.8$  & 16  &  7.0359e-02  &  1.149  & 32   & 1.5847e-01  & 1.192  \\
     $\alpha_2=0.2$  & 32  &  3.1263e-02  &  1.170  & 64   & 6.9235e-02  & 1.195  \\
                     & 64  &  1.3678e-02  &  1.193  & 128  & 3.0055e-02  & 1.204  \\
                     & 128 &  5.8410e-03  &  1.228  & 256  & 1.2871e-02  & 1.223  \\
    \midrule
                     & 8   &  8.7133e-01  &  -      & 16   & 1.9420e-00  & -       \\
     $\alpha_1=0.5$  & 16  &  3.0816e-01  &  1.499  & 32   & 6.8684e-01  & 1.499  \\
     $\alpha_2=0.5$  & 32  &  1.0901e-01  &  1.499  & 64   & 2.4295e-01  & 1.499  \\
                     & 64  &  3.8614e-02  &  1.497  & 128  & 8.5943e-02  & 1.499  \\
                     & 128 &  1.3721e-02  &  1.493  & 256  & 3.0454e-02  & 1.497   \\
    \bottomrule
  \end{tabular}
\end{table}

\subsection{Application II}\label{subsection7.2} Here, we consider another case. Denote the operators $\mathbf{B}_1=-\frac{1}{3}\frac{\partial^2}{\partial x^2}$ and $\mathbf{B}_2=-\frac{2}{3}\frac{\partial^2}{\partial x^2}$. Hence, with $\Omega=(0,1)$, problem \eqref{eq1.3} becomes
\begin{equation}\label{eq7.5}
	\begin{split}
           \frac{\partial u}{\partial t} & - \int_{0}^{t} \left[\frac{\beta_1(t-s)}{3}  +  \frac{2\beta_2(t-s)}{3}  \right] u_{xx}(x,s) ds  = f(x,t), \quad (x,t)\in \Omega\times (0,\infty),  \\
            &   u(x,0)  = u_0(x), \quad   x\in [0,1], \\
            &   u(0,t) = u(1,t) = 0, \quad  t \in(0,\infty).
    \end{split}
\end{equation}
Using \eqref{eq7.2} and the scheme \eqref{eq6.5}, we obtain the following fully discrete difference scheme
   \begin{equation}\label{eq7.6}
   \begin{split}
      \delta_t U^n_j  & - \frac{1}{3}\mathcal{Q}_{n-\frac{1}{2}}^{(1)}(\delta_x^2U_j) - \frac{2}{3}\mathcal{Q}_{n-\frac{1}{2}}^{(2)}(\delta_x^2U_j)   = f^{n-1/2}_j,
      \quad   n\geq 1, \quad   1\leq j \leq  \mathcal{M}-1,  \\
      & U_j^0=u_0(x_j), \quad 1\leq j \leq  \mathcal{M}-1, \\
      & U_{0}^n = U_{\mathcal{M}}^n = 0,  \quad  n\geq 1.
   \end{split}
   \end{equation}

   \vskip 1mm
     \textbf{Example 2.} Here, let the exact solution be given by \eqref{eq7.4}, thus the initial data $u_0(x)=\sin(\pi x)$ and the source term
   \begin{equation*}
   \begin{split}
      f(x,t) =& - \left(  \frac{t^{\alpha_*+1}}{\Gamma(\alpha_*+1)} + \frac{4\pi^2 t^{\alpha_*+\alpha_1+1}}{3\Gamma(\alpha_*+\alpha_1+2)}  + \frac{8\pi^2 t^{\alpha_*+\alpha_2+1}}{3\Gamma(\alpha_*+\alpha_2+2)}  \right)\sin 2\pi x \\
          &  + \left(  \frac{\pi^2 t^{\alpha_1}}{3\Gamma(\alpha_1+1)} +  \frac{2\pi^2 t^{\alpha_2}}{3\Gamma(\alpha_2+1)}  \right)\sin \pi x.
   \end{split}
   \end{equation*}

     \vskip 0.2mm
     Table \ref{tb2} lists the errors and corresponding time convergence rates for fully discrete difference scheme \eqref{eq7.6}. By testing different cases of $T=200$ and $T=500$, given suitable constant $c$ correspondingly, the results from Table \ref{tb2} illustrate that convergence rates in the time direction can reach $k^{1+\alpha_*}$, which is consistent with the theoretical finding.

\begin{table}
  \center
    \footnotesize
  \caption{The errors and temporal convergence orders with different values of $\alpha_*$, $T$ and $\mathcal{M}=1024$.} \label{tb2}
  \begin{tabular}{cccccccc}
    \toprule
  & & \multicolumn{2}{c}{$T=200$, $c=11$}& & \multicolumn{2}{c}{$T=500$, $c=14$} \\
   \cmidrule(r){3-4}  \cmidrule(r){6-7}
   \noalign{\smallskip}
   $\alpha_*=\min(\alpha_1,\alpha_2)$ & $N$ & $Error(k,h)$ & $Rate^t$ & $N$ & $Error(k,h)$ & $Rate^t$ \\
    \midrule
                     & 8   &  5.2929e-01  &  -      & 8    & 4.3881e-01  & -       \\
     $\alpha_1=0.3$  & 16  &  2.1487e-01  &  1.301  & 16   & 1.7810e-01  & 1.301  \\
     $\alpha_2=0.7$  & 32  &  8.7446e-02  &  1.297  & 32   & 7.2292e-02  & 1.301  \\
                     & 64  &  3.6154e-02  &  1.274  & 64   & 2.9381e-02  & 1.299  \\
                     & 128 &  1.6224e-02  &  1.156  & 128  & 1.2047e-02  & 1.286  \\
    \midrule
                     & 8   &  3.9401e-01  &  -      & 8    & 2.9178e-01  & -       \\
     $\alpha_1=0.8$  & 16  &  1.7343e-01  &  1.184  & 16   & 1.2908e-01  & 1.177  \\
     $\alpha_2=0.2$  & 32  &  7.6392e-02  &  1.183  & 32   & 5.6868e-02  & 1.183  \\
                     & 64  &  3.4240e-02  &  1.158  & 64   & 2.5022e-02  & 1.184  \\
                     & 128 &  1.6723e-02  &  1.034  & 128  & 1.1112e-02  & 1.171  \\
    \midrule
                     & 8   &  8.8039e-01  &  -      & 8    & 8.7593e-01  & -       \\
     $\alpha_1=0.5$  & 16  &  3.1135e-01  &  1.500  & 16   & 3.0969e-01  & 1.500  \\
     $\alpha_2=0.5$  & 32  &  1.1033e-01  &  1.497  & 32   & 1.0951e-01  & 1.500  \\
                     & 64  &  3.9704e-02  &  1.474  & 64   & 3.8761e-02  & 1.498  \\
                     & 128 &  1.5853e-02  &  1.325  & 128  & 1.3831e-02  & 1.487   \\
    \bottomrule
  \end{tabular}
\end{table}

\section{Concluding remarks}\label{section8}
In the present work, provided suitable hypotheses, the long-time global stability and convergence for time discretizations of VIDEs with multi-term non-smooth kernels were established by the z-transform and energy argument. However, in our analyses, the theoretical results were derived based on the decreasing exponentially weight norm. In the future work, we would like to prove the long-time global behavior of discrete scheme of problem \eqref{eq1.1} or  problem \eqref{eq1.3} while eliminating decreasing exponentially factors.

\section*{Conflict of Interest Statement}
The authors declare that they do not have any conflicts of interest.



\begin{thebibliography}{10}

      \bibitem {Cao} {\sc Y. Cao,  O. Nikan, Z. Avazzadeh}, A localized meshless technique for solving 2D nonlinear integro-differential equation with multi-term kernels, Appl. Numer. Math., 183 (2023), pp.~140--156.

    \bibitem {Carslaw} {\sc H. S. Carslaw, J. C. Jaeger}, {\em Conduction of Heat in Solids}, 2nd ed., Clarendon Press, Oxford, 1959.

    \bibitem {Cuesta1} {\sc E. Cuesta, C. Palencia}, {\em A fractional trapezoidal rule for integro-differential equations of fractional order in Banach spaces}, Appl. Numer. Math., 45 (2003), pp.~139--159.


    \bibitem {Hannsgen1} {\sc K. B. Hannsgen, R. L. Wheeler}, {\em Uniform $L^1$ Behavior in Classes of Integrodifferential Equations with Completely Monotonic Kernels}, SIAM J. Math. Anal., 15 (1984), pp.~579--594.

    \bibitem {Hu} {\sc S. Hu, W. Qiu, H. Chen},  {\em A backward Euler difference scheme for the integro-differential equations with the multi-term kernels}, Int. J. Comput. Math., 97 (2020), pp.~1254--1267.

    \bibitem {Jury} {\sc E. I. Jury}, {\em  Theory and application of the z-transform method}, John Wiey \& Sons, Inc., New York, 1964.


    \bibitem {Lopez-Marcos} {\sc J. C. L\'{o}pez-Marcos}, {\em  A difference scheme for a nonlinear partial integro-differential equation}, SIAM J. Numer. Anal., 27 (1990), pp.~20--31.

    \bibitem {Lubich}{\sc C. Lubich}, {\em  Discretized fractional calculus}, SIAM J. Math. Anal., 17 (1986), pp.~704--719.

    \bibitem {Lubich1}{\sc C. Lubich}, {\em  Convolution quadrature and discretized operational calculus. I}, Numer. Math., 52 (1988), pp.~129--145.

    \bibitem {MacCamy} {\sc R. C. MacCamy}, {\em  An integro-differential equation with application in heat flow}, Quart. Appl. Math., 35 (1977), pp.~1--19.

    \bibitem {McLean1} {\sc W. McLean, V. Thom\'{e}e}, {\em Numerical solution of an evolution equation with a positive-type memory term}, J. Austral Math. Soc. Ser. B, 35 (1993), pp.~23--70.

    \bibitem {McLean2} {\sc W. McLean, K. Mustapha}, {\em A second-order accurate numerical method for a fractional wave equation}, Numer. Math., 105 (2007), pp.~481--510.

    \bibitem {Nohel} {\sc J.A. Nohel, D.F. Shea}, {\em Frequency domain methods for Volterra equations}, Adv. Math., 22 (1976), pp.~278--304.

    \bibitem {Noren1} {\sc R. Noren}, {\em Uniform $L^1$ behavior in classes of integrodifferential equations with convex kernels}, J. Integr. Equ. Appl., 1 (1988), pp.~385--396.

    \bibitem {Noren2} {\sc R. Noren}, {\em Uniform $L^1$ behavior in a class of linear Volterra equations}, Quart. Appl. Math., 47 (1989), pp.~547--554.

    \bibitem {Pruss} {\sc J. Pr\"{u}ss}, {\em Evolutionary integral equations and applications}, Monographs Mathematics, Vol. 87, Birkh\"{a}user Verlag, Basel, 1993.

    \bibitem {Qiu1} {\sc  W. Qiu, D. Xu, H. Chen}, {\em  A formally second-order BDF finite difference scheme for the integro-differential equations with the multi-term kernels}, Int. J. Comput. Math., 97 (2020), pp.~2055--2073.

    \bibitem {Qiu2} {\sc  W. Qiu, D. Xu, J. Guo}, {\em  Numerical solution of the fourth-order partial integro-differential equation with multi-term kernels by the Sinc-collocation method based on the double exponential transformation}, Appl. Math. Comput., 392 (2021), p.~125693.

    \bibitem {Qiu3} {\sc  W. Qiu}, {\em  Optimal error estimate of accurate second-order scheme for Volterra integrodifferential equations with tempered multi-term kernels}, arXiv preprint arXiv:2209.00229, 2022.

   \bibitem {Renardy} {\sc M. Renardy, W. J. Hrusa, J. A. Nohel}, {\em Mathematical problem in viscoelasticity}, Longman, London, 1987.


   \bibitem {Xu1} {\sc D. Xu}, {\em The time discretization in classes of integro-differential equations with completely monotonic kernels: Weighted asymptotic stability}, Sci. China Math., 56 (2013), pp.~395--424.

   \bibitem {Xu2} {\sc D. Xu}, {\em The time discretization in classes of integro‐differential equations with completely monotonic kernels: Weighted asymptotic convergence}, Numer. Methods Part. Differ. Equ., 32 (2016), pp.~896--935.


   \bibitem {Xu3} {\sc D. Xu}, {\em Numerical asymptotic stability for the integro-differential equations with the multi-term kernels}, Appl. Math. Comput., 309 (2017), pp.~107--132.

   \bibitem {Xu4} {\sc D. Xu}, {\em Second-order difference approximations for Volterra equations with the completely monotonic kernels}, Numer. Algorithms, 81 (2019), pp.~1003--1041.


    \bibitem {Yan} {\sc Y. Yan, G. Fairweather}, {\em Orthogonal spline collocation methods for some partial integrodifferential equations}, SIAM J. Numer. Anal., 29 (1992), pp.~755--768.





\end{thebibliography}
\end{document}